\documentclass{amsart}

\usepackage{verbatim}
 \usepackage{amsmath, amssymb, amsfonts, euscript, enumerate,latexsym}
\usepackage{amsthm}

\usepackage{txfonts}
\usepackage{fancyhdr}
 \usepackage{mathrsfs}
\usepackage{mathtools}
\usepackage{epsfig,subfigure}
\usepackage{graphicx}
\usepackage{color,wrapfig}
\usepackage{txfonts}
 
 \usepackage{hyperref}

\newcommand{\R}{{\mathbb R}}
\newcommand{\N}{{\mathbb N}}
\newcommand{\Z}{{\mathbb Z}}

\newcommand{\cG}{{\mathcal G}}

\newcommand{\cA}{{\mathcal A}}

\newcommand{\cH}{{\mathcal H}}
\newcommand{\cD}{{\mathcal D}}

\newcommand{\cM}{{\mathcal M}}

\newcommand{\cF}{{\mathcal F}}
\newcommand{\cS}{{\mathcal S}}

\newcommand{\cX}{{\mathcal X}}
\newcommand{\cQ}{{\mathcal Q}}

\newcommand{\sH}{{\mathscr H}}
\newcommand{\sS}{{\mathscr S}}
\newcommand{\sG}{{\mathscr G}}

\newcommand{\e}{\epsilon}
\newcommand{\ve}{\varepsilon}
\newcommand{\al}{\alpha}

\newcommand{\p}{\partial}

\newcommand{\Cut}{\operatorname{Cut}}
\newcommand{\Jac}{\operatorname{Jac}}

\newcommand{\Ric}{\operatorname{Ric}}
\newcommand{\Sec}{\operatorname{Sec}}
\newcommand{\Sym}{\operatorname{Sym}}

\newcommand{\D}{\nabla}

\newcommand{\La}{\Delta}

\newcommand{\vol}{\operatorname{Vol}}
\newcommand{\diam}{\operatorname{diam}}

\newtheorem{thm}{Theorem}[section]
\newtheorem{lemma}[thm]{Lemma}
\newtheorem{cor}[thm]{Corollary}
\newtheorem{remark}[thm]{Remark}
\newtheorem{prop}[thm]{Proposition}
\newtheorem{definition}[thm]{Definition}

\newtheorem*{notation}{Notation}

\theoremstyle{definition}

\begin{document}
\title[Second derivative estimates for elliptic operators on Riemannian manifolds]{   Second derivative estimates for    uniformly   elliptic operators  on Riemannian manifolds}
 
\author[Soojung Kim]{Soojung Kim}
\address{Soojung Kim :
National Institue for Mathematical Sciences, 
70 Yuseong-daero 1689 beon-gil, Yuseong-gu,   
Daejeon, 306-390,  Republic of Korea  }
\email{soojung26@nims.re.kr, soojung26@gmail.com }

%\keywords{}
%\subjclass{Primary 35K55, 35K65}

\maketitle
\begin{abstract}
In    this paper, we   obtain   a  uniform $W^{2,\ve}$-estimate  of solutions to the fully nonlinear uniformly  elliptic equations 
%$$F(D^2u,x)=0\quad\mbox{in $B_R(z_0).$}$$
%in the class $\cS^* (\lambda,\Lambda, f)$  
on Riemannian manifolds with a lower bound of sectional curvature using the ABP method.  
 \end{abstract}
%{\bf Contents:}\\

\tableofcontents 
%\newpage

%%%%%%%%%%%%%%%%%%%%%%%%%%%%%%%%%%%%%%%%%%%%%%%%%%%%%%%%%%%%%%%%%%%%%%%%%%%%%%%%%%%%%%%%%%%%%%%%%%%%%%%\section
%%%%%%%%%%%%%%%%%%%%%%%%%%%%%%%%%%%%%%%%%%%%%%%%%%%%%%%%%%%%%%%%%%%%%%%%%%%%%%%%%%%%%%%%%%%%%%%%%%%%%%
\section{Introduction}%\label{sec-intro}

 We   study    regularity  estimates for  solutions to   a class of   the fully nonlinear uniformly  elliptic equations  
 \begin{equation}\label{eq-main}
F(D^2u,x)=f\quad\mbox{in $B_R(z_0)\subset M.$}
\end{equation}
%in the class $\cS^* (\lambda,\Lambda, f)$  
on a complete Riemannian manifold $M,$ where the operator $F$ satisfies the hypothesis  \eqref{Hypo1}.  Under the assumption that   
sectional  curvature    of $M$ is nonnegative,
the Krylov-Safonov Harnack estimate \cite{KS} was initiated by Cabr\'e in his paper \cite{Ca}, where 
  a priori  global Harnack inequality for  linear elliptic equations 
  was established by  obtaining the Aleksandrov-Bakelman-Pucci (ABP) estimate on $M.$ %\cite{Al,B,P} on $M.$  
   Later, Kim \cite{K} improved Cabr\'e's result removing the sectional curvature assumption and imposing the certain conditions on the squared distance  function. Recently, Wang and Zhang \cite{WZ} proved a version of the ABP estimate on $M$ with a lower bound of Ricci curvature, and  hence a locally uniform Harnack inequality for nonlinear  elliptic operators on $M$ provided that  the sectional curvature is bounded from below.  A priori Harnack estimates   have been extended in \cite{KL} for viscosity solutions       using the regularization of  JensenÕs sup-convolution on Riemannian manifolds.  
  The H\"older continuity is obtained as  an immediate consequence of the Harnack inequality.     In \cite{KKL, KL}, the parabolic Harnack inequality  and the ABP-Krylov-Tso type estimate were  established    in the Riemannian setting. 
  % In [KL], the authors proved Harnack inequalities for viscosity solutions by making use of regularization via JensenÕs sup- and inf- convolutions on Riemannian manifolds

In this paper, we   investigate    a  uniform $W^{2,\ve}$-regularity (for some $\varepsilon>0$) of solutions to \eqref{eq-main} on Riemannian manifolds with a lower bound of sectional curvature. %via the ABP method.     %i n nondivergence form.  
%in the class $\cS^* (\lambda,\Lambda, f)$  
%on Riemannian manifolds with a lower bound of sectional curvature. %via the ABP method.     
In the Euclidean space,  a uniform $W^{2,\varepsilon}$-estimate (for some        $\varepsilon>0$)  for   linear, nondivergent  elliptic operators    with measurable coefficients      was    first discovered by Lin \cite{L}.   It is known that for any $p\geq 1,$    a uniform    $W^{2,p} $-estimate  for  uniformly  elliptic equations  with measurable coefficients is not valid; see \cite{PT,U}.   
    In  \cite{C}    \cite[Chapter 7]{CC}, Caffarelli  dealt with $W^{2,\varepsilon}$-estimates for fully nonlinear elliptic operators, where 
the ABP estimate    is a  keystone   in the proof   %of   $W^{2,\epsilon}$-estimates for fully nonlinear elliptic equations, 
together with    the Calder\'on-Zygmund technique.  The  ABP estimate   proved by Aleksandrov, Bakelman, and Pucci  in sixties   
has played a crucial role in the Krylov-Safonov theory for   nondivergent  elliptic equations  with measurable coefficients, and in the development of  the regularity theory for fully nonlinear equations.    

Making use of the ABP type estimate on Riemannian manifolds, we  follow Caffarelli's approach to extend   $W^{2,\varepsilon}$-estimates  for fully nonlinear   elliptic operators on Riemannian manifolds under the assumption that sectional curvature is  bounded from below.    
% In the Riemannian setting, i
It can be   checked  that the %for the $W^{2,\epsilon}$-estimate  
a straightforward adaptation of the Euclidean method     yields    the $W^{2,\varepsilon}$-estimate on   Hadamard manifolds which are    complete and simply-connected Riemannian manifolds with nonpositive sectional curvature everywhere.
In general, it is not applicable directly due to  the existence of the cut locus. Indeed, 
%On general  Riemannian manifolds,  due to  the existence of the cut locus,  
%    a straightforward adaptation is not  possible  % prevents us from  using the Euclidean approach in    \cite[Chapter 7]{CC}  since 
%we need to adapt     the iteration scheme since 
   it is difficult to use the  squared distance  functions       as   global  test  functions as in the Euclidean case.
  %   it is difficult to see that  the  squared distance  functions    act  as   global  test  functions as in the Euclidean case. % on  manifolds.  %for  smooth solutions . 
%  To establish uniform $W^{2,\epsilon}$-estimates on Riemannian manifolds with sectional curvature bounded from below, 
To proceed with  the ABP method,  %   with  the help of the  Calder\'on-Zygmund technique,
 we introduce the notion of the  special contact set in Definition \ref{def-special-contact -set} which 
 consists of the points where  the solution has a global tangent function which is  a sum of   the scale invariant  barrier functions in Lemma \ref{lem-barrier} and squared distance functions.  With  the help of the  Calder\'on-Zygmund technique,  the notion of the special contact set 
  enables  us to     employ  an iterative procedure
  using    the ABP type estimate   in Proposition \ref{prop-decay-est-supersol}.      
  % The  special contact set consists of the points where  the solution has a global tangent function which is  a sum of   the scale invariant  barrier functions in Lemma \ref{lem-barrier} and squared distance functions. Therefore, 
  Therefore we  deduce  a (locally) uniform $W^{2,\varepsilon }$-estimate for    a class of  solutions to  the fully nonlinear   uniformly elliptic equations  in Definition \ref{def-class-solutions} which  includes    the solutions to \eqref{eq-main}.  %as follows. %   by  obtaining  the decay estimate for the distribution function of the norm of  Hessian of solutions in Corollary \ref{cor-W2p}. 

\begin{thm}[$W^{2,\ve}$-estimate] \label{thm-W2epsilon}
Let $M$ be a complete Riemannian manifold with the sectional curvature bounded from below by $-\kappa$ 
  for $\kappa\geq0.$ 
Let $0<R\leq R_0$ and $x_0\in M$ and $f\in L^{n\eta}\left(B_{2R}(x_0)\right)$ for $ \eta:=1+\log_2\cosh (4\sqrt\kappa R_0).$  
There exist   uniform constant $\ve>0$ and $C>0$ such that if    a smooth function $u$ belongs to   $ \cS^*\left(\lambda,\Lambda,f\right)$ in % solution to $F(D^2u)=f$ in 
$B_{2R}(x_0),$  
then we have that   $u\in W^{2,\ve}\left(B_R(x_0)\right)$ with the estimate
$$\left(\fint_{B_{R}(x_0)}  |u|^\ve+ \left|R \D u\right|^{\ve}+ \left|R^2D^2u\right|^\ve  \right)^{\frac{1}{\ve}} \leq C \left\{ \|u\|_{L^{\infty}\left(B_{2R}(x_0)\right)}+\left(\fint_{B_{2R}(x_0)} |R^2f|^{n\eta}\right)^{\frac{1}{n\eta}}\right\},$$
where  $\ve>0$ and $C>0$ depend only on $n,\lambda,\Lambda,$ and $\sqrt{\kappa}R_0,$ and 
we denote $\fint_Q f := \frac{1}{\vol(Q)}\int_Qf\,d\vol.$
\end{thm}
  
When a Riemannian manifold has nonnegative sectional curvature, i.e.,  $\kappa=0,$ the   $W^{2,\varepsilon}$-estimate is global,    and  depends only on   dimension $n,$ and the ellipticity constants $\lambda,$ and $\Lambda.$
 
% The paper is organized as follows. In Section 2, we     recall some of the known results on Riemannian geometry and the notion of uniformly elliptic operators. Section 3 is devoted to  the   proof of Theorem \ref{thm-W2epsilon}. 
 %the   $W^{2,\ve}$-estimate of   solutions to uniformly elliptic equations.  
 
 %hat are used in the paper. In Section 3, we briefly recall the notion of viscosity solutions to fully nonlinear elliptic and parabolic equations. Section 4 is devoted to the detailed proof of the uniform $W^{2,\ve}$-estimate of   solutions to uniformly elliptic equations.  
%%%%%%%%%%%%%%%%%%%%%%%%%%%%%%%%%%%%%%%%%%%%%%%%%%%%%%%%%%%%%%%%%%%%%%
\section{Preliminaries}%\label{sec-pre}
%%%%%%%%%%%%%%%%%%%%%%%%%%%%%%%%%%%%%%%%%%%%%%%%%%%%%%%%%%%%%%%%%%%%%%%%%%%%%%%%%%%%%%%%%%%%%%%%%%%%%%% section
%%%%%%%%%%%%%%%%%%%%%%%%%%%%%%%%%%%%%%%%%%%%%%%%%%%%%%%%%%%%%%%%%%%%%%%%%%%%%%%%%%%%%%%%%%%%%%%%%%%%%%
%\subsection{Riemannian geometry}
%%%%%%%%%%%%%%%%%%%%%%%%%%%%%%%%%%%%%%%%%%%%%%%%%%%%%%%%%%%%%%%%%%%%%%%%%%%%%%%%%%%%%%%%%%%%%%%%%%%%%%% section
%%%%%%%%%%%%%%%%%%%%%%%%%%%%%%%%%%%%%%%%%%%%%%%%%%%%%%%%%%%%%%%%%%%%%%%%%%%%%%%%%%%%%%%%%%%%%%%%%%%%%%

Throughout  this paper, let  $(M, g)$ be a smooth, complete Riemannian manifold of dimension $n$,  where $g$ is the  Riemannian metric. A Riemannian metric defines a scalar product and a norm on each tangent space, i.e., $\langle X,Y\rangle_x:=g_x(X,Y)$  and $|X|_x^2:=\langle X,X\rangle_x$ for $X,Y\in T_xM$,  where $T_x M$ is the tangent space at $x\in M$.  Let $d(\cdot,\cdot)$ be the Riemannian distance   on $M$.  For a given point $y\in M$,   $d_y(x)$  denotes  the distance   to $x$ from $y$,  i.e., $d_y(x):=d(x,y)$.   
A Riemannian manifold is equipped with     the Riemannian measure $\vol=\vol_g$ on $M$ which  is     denoted  by   $|\cdot|$  for simplicity.

For a smooth function  $u:M\to\R,$  the gradient $\D u$ of $u$ is defined by 
$$\langle\D u, X \rangle := du(X) $$
for any vector field $X$ on $M,$ where $du:TM\to\R$ is the differential of $u.$
The Hessian $D^2u$ of $u$  is defined as 
$$D^2u  \,  (X, Y):=  \left\langle\nabla_X \nabla u, Y\right\rangle,$$
 for any vector fields $X, Y$ on $M,$ where $\D$ denotes the Riemannian connection of $M.$   We observe that the Hessian $D^2u$ is a  symmetric 2-tensor over $M,$   and $D^2u(X,Y)$  at $x\in M$ depends only on  the values $X, Y$ at $x,$ and  $u$   in a small neighborhood of $x.$ 
 By the metric,  the Hessian of $u$ at $x$  is canonically identified with  
 a  symmetric endomorphism of $T_xM$: 
 $$
D^2u(x)  \cdot X=  \nabla_X \nabla u,\quad\forall X\in T_xM. 
$$
We will   write $D^2u(x)  \,  (X, Y)=\left\langle D^2u (x)\cdot X,\, Y\right\rangle$ for $X\in T_xM.$ % as a quadratic form on $T_xM.$ 
In terms of local coordinates  $\left(x^i\right) $ of $M,$ 
the components of $D^2u$  are written by 
$$\left(D^2u\right)_{ij}= \frac{\p^2u}{\p x^i\p x^j} -\Gamma_{ij}^k\frac{\p u}{\p x^k} $$
  where $\left\{\Gamma^k_{ij}\right\}$ are the Christoffel symbols of the Riemannian connection $\D$ of $M.$  Here and in what follows, we adopt  the Einstein summation convention.  % is adopted. 
In  \cite{H} (see   \cite{Au}), the norm of  the Hessian of $u,$ $|D^2u|$ is defined  in local coordinates  by
$$|D^2u|^2:=g^{ir}g^{js}\left(D^2u\right)_{ij}\left(D^2u\right)_{rs}.$$

Denote by $R$   the Riemannain curvature tensor   defined as
$$ R(X, Y)Z = \D_X \D_YZ- \D_Y\D_XZ  -\D_{ [X,Y ]}Z$$
for any vector fields $X, Y,Z$ on $M.$ 
%where $\D$ denotes the Riemannian connection of $M.$ 
For 
 two linearly independent  vectors $X,Y\in T_xM,$       the sectional curvature of the
plane generated by $X$ and $Y$  is defined as 
 $$\Sec(X,Y):=\frac{\langle R(X, Y)X,Y\rangle}{|X|^2|Y|^2-\langle X,Y\rangle^2}.$$
%If $\Sec(X,Y) \geq \kappa$ for all $X,Y \in T_xM$ and $x\in M$ with   $\kappa\in\R,$ we write $\Sec\geq \kappa$ on $M$ for simplicity. 
 The   Ricci curvature  tensor denoted by $\Ric$ is defined as  follows: for a unit vector $X\in T_xM$ and    an orthonormal basis $\{X,e_2,\cdots,e_n\}$ of $T_xM,$ $$\Ric(X,X):=\sum_{j=2}^n \Sec(X,e_j).$$  
  l,   $\Ric\geq \kappa$ on $M\,\,\,(\kappa\in\R)$ stands  for   $\Ric_x \geq \kappa g_x $ for all $x\in M.$ 
  We refer to \cite{D,Le} for Riemannian geometry.   %We note that if $\Sec_x\geq \kappa\,\,\,(\kappa\in\R)$ for $x\in M,$  then   $\Ric_x\geq (n-1)\kappa.$

%\begin{itemize}
%\item   Interpretation
%\end{itemize}

\begin{comment}

For  n-dimesional   Riemannian manifolds $M$ and $N,$    
let $\phi:M\to N$ be smooth.  
The Jacobian of $\phi$ is  the absolute value of determinant of the differential $d \phi$,  i.e., $$\Jac\phi(x):=|\det d\phi(x)|\quad\mbox{for $x\in M$}. $$  
For any smooth function $\phi:M \to M$ and any measurable set $E\subset M$, we have
\begin{equation*}
\int_E \Jac\phi (x)\, d\vol(x)=\int_{M} \cH^0[E\cap\phi^{-1}(y)] \, d\vol(y),\end{equation*}
where $\cH^0$ is the counting measure.
\end{comment}
 
Assuming the Ricci curvature to be  bounded from below,    Bishop-Gromov's volume comparison theorem  says that the volume of balls does not increase faster than the volume of balls in the model space (see \cite{V} for instance).  In particular,   the volume     comparison   implies the following   (locally uniform) volume doubling property. 
 %(see \cite[Chapter 2]{L} for instance).
 %We have the following doubling property and refer to \cite[Chapter 14]{V}.t
 \begin{thm}[Bishop-Gromov] \label{thm-BG}
Assume that   $\Ric 
  \geq -(n-1)\kappa$  on $M$ for $\kappa\geq0.$ 
 % \begin{itemize}
  %\item[(i) ] For $x,y\in M$ with $x\notin\Cut(y),$ 
    %$$\La d_y \leq  (n-1)\sqrt{k}\coth\left(\sqrt{k}d_y(x)\right).$$
  %\item[(ii)]{\rm(Bishop-Gromov)}
 For any $0<r < R,$ we have 
  \begin{equation}\label{eq-doubling}
  \frac{\vol(B_{2r}(z))}{\vol(B_r(z))}\leq 2^n\cosh^{n-1}\left(2\sqrt{\kappa}R\right)=:\cD,
  \end{equation}
    where    $\cD $  is the so-called doubling constant. 
  %\end{itemize}
  \end{thm} 
One can check  that the doubling property \eqref{eq-doubling}  yields that for any $0<r< R< R_0,$  
  \begin{equation*}%\label{eq-doubling-decay}
  \frac{\vol\left(B_{R}(z)\right)}{\vol\left(B_r(z)\right)}\leq  \cD\, \left(\frac{R}{r}\right)^{\log_2\cD},
  \end{equation*}
  for $\cD:= 2^n\cosh^{n-1}\left(2\sqrt{ \kappa}R_0\right).$
    According to the volume   comparison, it is easy to prove the following lemma. %It turns out that 
Below and hereafter, we denote $\fint_Q f := \frac{1}{ |Q|}\int_Qf\,d\vol.$

\begin{lemma}\label{lem-weight-int-ellip}
Assume that   for any  $z\in M$ and $0 <r < 2R_0,$ there exists   a doubling constant $\cD>0$ such that
$$\vol\left(B_{2r}(z)\right)\leq \cD \vol\left(B_r(z)\right).$$
Then we have that  for any $B_{r}(y)\subset B_R(z)$ with $0<r <  R <R_0, $  
\begin{equation}\label{eq-weight-int-ellip}
\left\{\fint_{B_r(y)}\left|r^2f\right|^{n\eta}\right\}^{\frac{1}{n\eta}}\leq  2 \left\{\fint_{B_R(z)}\left|R^2f\right|^{n\eta}\right\}^{\frac{1}{n\eta}};\qquad \eta:= \frac{1}{n}\log_2\cD.
\end{equation} 
%where $\eta:= \frac{1}{n}\log_2\cD.$  
In particular, if  the sectional curvature of $M$ is bounded from below by $-\kappa$ ($\kappa\geq0$),    then \eqref{eq-weight-int-ellip} holds for $\eta := 1+\log_2 \cosh(4\sqrt{\kappa}R_0). $ % $\eta := 1+\frac{n-1}{n}\log_2 \cosh(4\sqrt{\kappa}R_0). $
\end{lemma}

A Hessian bound for the squared distance function is the following lemma which is proved  in \cite[Lemma 3.12]{CMS}  making use of  the formula for the second variation of energy  provided  that the sectional curvature is bounded from below.  %According to the local characterization of semi-concavity combined with Lemma \ref{lem-hess-dist-sqrd},     $d_y^2$ is semi-concave  on  a bounded open set $\Omega\subset M$ for any $y\in M,$ provided that the sectional curvature of $M$ is bounded from below. %, and     an open set $\Omega \subset M$   is  bounded. 
%the closure $\overline\Omega$ of an open set $\Omega \subset M$   is compact.    

%Under the assumption on the sectional curvature, a Hessian bound for the squared distance function in  the following lemma   is quoted from   \cite[Lemma 3.12]{CMS}. %, and     an open set $\Omega \subset M$   is  bounded. 
%the closure $\overline\Omega$ of an open set $\Omega \subset M$   is compact.    

 \begin{lemma}\label{lem-hess-dist-sqrd}
 Let $x, y\in M.$  If $\Sec
  \geq -\kappa\,\, ( \kappa\geq0)$ along a minimizing geodesic  joining $x$ to $y,$ then for any unit vector $X\in T_xM,$ % with $|X|=1,$
  \begin{equation*}%\label{eq-hess-dist-sqrd}
  \limsup_{r\to0}\frac{d_y^2\left(\exp_xrX\right)+d_y^2\left(\exp_x -rX\right)-2d^2_y(x)}{r^2}\leq 2
  \sH\left( \sqrt{\kappa}d_y(x)\right),
%  \sqrt{\kappa}d_y(x)\coth\left(\sqrt{\kappa}d_y(x)\right)
  \end{equation*} 
  where $\sH(t):=t\coth(t)$ for $t\geq0. $
  \end{lemma}
%  We denote 
  %$$\sH(t):=t\coth(t),\quad \sS(t):= \sinh(t)/t\quad\forall t>0, $$ with $\sH(0)=0$ and $\sS(0):=0, $  
  
It is not difficult to obtain the following corollary 
 modifying the proof of  \cite[Lemma 3.12]{CMS} with the help of  the  monotonicity of a composed function $\psi.$
 \begin{cor}\label{cor-hess-dist-sqrd} 
 Let $\psi:\R\to\R$ be a smooth, even  function such that $\psi'(s)\geq0$ for $s\in[0,+\infty).$  
 Let $x, y\in M.$  If $\Sec
  \geq -\kappa\,\, ( \kappa\geq0)$ along a minimizing geodesic  joining $x$ to $y,$ then for any unit vector $X\in T_xM,$  
  \begin{align*}
  \limsup_{r\to0}\frac{\psi\left(d_y^2\left(\exp_xrX\right)\right)+\psi\left(d_y^2\left(\exp_x -rX\right)\right)-2\psi\left(d^2_y(x)\right)}{r^2}\\
  \leq 2\sqrt{\kappa}d_y(x)\coth\left(\sqrt{\kappa}d_y(x)\right) \psi'\left(d_{y}^2(x) \right)+4d_{y}^2(x) \psi''\left(d_{y}^2(x) \right).
  \end{align*}
  \end{cor}
   
    According to   \cite[Proposition 2.5]{CMS}, 
    the cut locus of $y\in M$ is characterized   as the set of points at which   the squared distance function $d_{y}^2$ is not  smooth. We state it as a lemma  which  says   that the  semi-convexity of the squared distance functions  fails at the cut locus.
 
   \begin{lemma}\label{lem-dist-sqrd-cut}
 Let $x, y\in M.$  If $x\in \Cut(y),$ then  there is  a unit vector $X\in T_xM$ such that 
 $$\liminf_{r\to0}\frac{d_y^2\left(\exp_xrX\right)+d_y^2\left(\exp_x -rX\right)-2d^2_y(x)}{r^2} =-\infty.$$
  \end{lemma}

  %\begin{proof}
  
  %\end{proof}
  
  \begin{comment}
The following  result  from Bangert is an extension of Aleksandrov's second differentiability theorem  that a convex function has second derivatives almost everywhere in the Euclidean space \cite{A} (see also \cite[Chapter 14]{V}) .  
\begin{thm}[Aleksandrov-Bangert, \cite{B}]\label{thm-AB}
Let $\Omega\subset M$ be an open set  and let $\phi:\Omega\to \R$   be semi-concave. Then  for almost every $x\in\Omega, $ $\phi$ is differentiable at  $x,$ and   there exists a symmetric operator $A(x):T_xM\to T_xM$ %. In particular, for $x\in\Omega$ and $\xi\in T_xM,$
characterized by any one of the two equivalent properties:
%in the sense that   for   $\xi\in T_xM,,$
\begin{enumerate}[(a)]
\item  for   $\xi\in T_xM,\,\,\,$ $A(x)\cdot\xi=\D_\xi\D\phi(x),$ \\
\item
%\begin{equation*}%\label{eq-taylor-exp}
$\phi(\exp_x \xi)=\phi(x)+\left\langle\D\phi(x),\xi\right\rangle+\frac{1}{2}\left\langle A(x)\cdot \xi,\xi\right\rangle+o\left(|\xi|^2\right)\quad\mbox{ as $\xi\to0.$}
$%\end{equation*}\\
\end{enumerate}
%The operator $A(x)$  is denoted by $D^2\phi(x). $  % The  Hessian $D^2\phi(x)$  can be viewed either as a symmetric quadratic form on $T_x M$, or  as a symmetric operator from $T_x M$ to itself. 
%When no confusion is possible, 
The operator $A(x)$   and its associated  symmetric bilinear from  on $T_xM$  are  denoted by $D^2\phi(x)$ and  called the Hessian of $\phi$ at $x$ when   no confusion is possible. 
\end{thm}
\end{comment}
 
%We end this subsection by stating  the area formula.
Let $M$ and $N$ be   Riemannian manifolds of dimension $n$ and  $\phi:M\to N$ be smooth.  
The Jacobian of $\phi$ is defined as %the absolute value of determinant of the differential $d \phi,$  that is, 
$$\Jac\phi(x):=|\det d\phi(x)|\quad\mbox{for $x\in M$}. $$  
Now we state  the area formula, which can be  proved by using    the area formula in Euclidean space and  a partition of unity: 
%\begin{lemma}[Area formula]
 For a Lipschitz continuous  function $\phi:M\to M,$ and a  measurable set $E\subset M$, we have
\begin{equation*}
\int_E \Jac\phi \,d\vol=\int_{M} \cH^0\left(E\cap\phi^{-1}(y)\right) \,d\vol (y),\end{equation*}
where $\cH^0$ is the counting measure. 

%\end{lemma}

Now, we present a  standard theorem called the weak type $(1,1)$ estimate   in the classical harmonic analysis, which will be used in the proof of our key estimate in Proposition \ref{prop-decay-est-supersol}. The proof relies  on the volume doubling property and Vitali's covering lemma;   see \cite[Chapter 1]{St} for details.
\begin{lemma}[Weak type $(1,1)$]\label{lem-max-ft}
 Assume that  $\Sec\geq -\kappa$   for $\kappa\geq0$  
and   let $x_0\in M$ and $0<R\leq R_0.$ % assume that $  B_{R}(z_0)\subset B_{R_0}(x_0).$  
For  an integrable function  $f$ in $B_{R}(x_0),$   the Hardy-Littlewood maximal operator $m_{B_{R}(x_0)}$ over $B_{R}(x_0)$ is defined  as 
$$m_{B_{R}(x_0)} \left(f\right)(z):=\sup_{z\in B_{r}(x)\subset B_{R}(x_0)}\fint_{B_r(x)}|f|\, d\vol.$$
Then  there exists a uniform constant $C_1:=2\cD^{1+\log_25}$ for   the doubling constant $\cD:= 2^n\cosh^{n-1}\left(2\sqrt{ \kappa}R_0\right)$   such that
$$\left|\left\{ z\in B_{R}(x_0): m_{B_{R}(x_0)} \left(f\right) (z)\geq  h\right\}\right|\leq \frac{C_1}{h}||f||_{L^1\left(B_R(x_0)\right)}\quad\forall h>0.$$ 
\end{lemma}
For the rest of this section,  we recall the concept of the uniformly elliptic operators. % and a class of solutions to
Let   $\Sym TM$ be the bundle of symmetric 2-tensors over $M.$
%where we may assume that $F(0)=0$ by replacing $f$ by $f-F(0).$
 %  Throughout this paper,  a 
An  operator  $F : \Sym TM\times M \rightarrow \R$ is said    to  be  uniformly elliptic with the so-called ellipticity constants $0<\lambda\leq\Lambda,$    if we have  
 \begin{equation}\label{Hypo1}\tag{H1}
   \lambda\,\mathrm{trace}(P_x)\leq F(S_x+P_x,x)-F(S_x,x)\leq \Lambda\, \mathrm{trace}(P_x),\quad\forall x\in M
 \end{equation}
 for  any $S\in \Sym TM,$ and   positive semi-definite $P\in\Sym TM.$ % where $\Sym TM$ is the bundle of symmetric 2-tensors over $M.$
% We also assume that $.$
% \end{center}
% \end{enumerate} 
  As extremal cases of the uniformly elliptic operators,   Pucci's   operators   are defined as follows:   %\cite{CC} for the proof.
%\begin{definition}[Pucci's extremal operator]\label{def-pucci}
 % for $0<\lambda\leq \Lambda$ (called ellipticity constants),  Pucci's extremal operators are defined as follows: 
 for any $x\in M,$ and  $S_x\in\Sym TM_x,$
\begin{align*}
\cM^+_{\lambda,\Lambda}(S_x):= \lambda\sum_{e_i<0}e_i+\Lambda\sum_{e_i>0}e_i,\\
\cM^-_{\lambda,\Lambda}(S_x):= \Lambda\sum_{e_i<0}e_i+\lambda\sum_{e_i>0}e_i,
\end{align*}
where $e_i=e_i(S_x)$ are the eigenvalues of $S_x.$ We will usually drop the subscripts $\lambda$ and $\Lambda,$ and write  $\cM^\pm.$
%\end{definition}
 When  $\lambda=\Lambda=1,$ $\cM^{\pm}$   simply coincide with the trace operator, that is,   $\cM^{\pm}(D^2u)=\La u$. We observe  that   \eqref{Hypo1}   is equivalent to the following: 
 for any $S, P\in\Sym TM,$
    \begin{equation*}%\label{Hypo2}\tag{H1'}
   \cM^- (P_x) \leq F(S_x+P_x,x)-F(S_x,x) \leq\cM^+ (P_x) \quad\forall x\in M.
 \end{equation*} 
The following lemma is concerned with basic properties of the Pucci operators;  see \cite[Chapter 2]{CC} for details.  
\begin{lemma}%\label{lem-prop-pucci-op}
Let $\Sym(n)$ denote the set of $n\times n$ symmetric matrices.  For $S,P\in \Sym(n),$ the followings hold: 
\begin{enumerate}[(i)]
\item 
%\begin{align*} 
 $$\cM^+ (S)=\sup_{A\in\cA_{\lambda,\Lambda}}\mathrm{trace} (AS),\quad\mbox{and}\quad  \cM^-(S)=\inf_{A\in\cA_{\lambda,\Lambda}} \mathrm{trace} (AS),$$
%\end{align*}
where $\cA_{\lambda,\Lambda}$ consists of  positive definite symmetric matrices in $\Sym(n),$   whose eigenvalues  lie in $[\lambda,\Lambda]$.
\item  $\cM^-(-S)=-\cM^+(S)$.
\item $\cM^{\pm}(t S)=t\cM^{\pm}(S)$ for any $t\geq0.$
\item $\cM^-(S+P)\leq  \cM^-(S)+\cM^+(P)\leq \cM^+(S+P)\leq \cM^+(S)+\cM^+(P)$.
\end{enumerate}
\end{lemma}

In order to   study a uniform $W^{2,\varepsilon}$-regularity for a    
class of  uniformly elliptic  equations   such as \eqref{eq-main}, 
  we introduce  a  more general class  of solutions to the uniformly elliptic equations    by using the Pucci operators  as    in \cite[Chapter 2]{CC}.  We notice that  the solution to the fully nonlinear  elliptic equation \eqref{eq-main}   belongs to the class $\cS^*\left(\lambda,\Lambda, f-F(0,\cdot)\right)$ below.

\begin{definition}\label{def-class-solutions}
%Let $f$ be  a continuous function in an open set 
Let $\Omega\subset M$  be open, and  
let $0<\lambda\leq \Lambda. $ We define a class of supersolutions  $
\overline \cS\left(\lambda,\Lambda,f\right)$ by  the set of     $u\in C^2(\Omega)$  satisfying
$$\cM^-(D^2u)\leq f\quad\mbox{ a.e. in $\Omega.$} $$ 
Similarly,  a class of subsolutions $\underline\cS\left(\lambda,\Lambda,f\right)$ is defined as the set of     $u\in C^2(\Omega)$ such that 
$$\cM^+(D^2u)\geq f\quad\mbox{a.e. in $\Omega.$} $$ 
We also define $$\cS^*\left(\lambda,\Lambda,f\right):=\overline \cS\left(\lambda,\Lambda,|f|\right)\cap \underline \cS\left(\lambda,\Lambda,-|f|\right).$$ We write shortly $\overline \cS(f),\underline \cS(f), $and $  \cS^*(f)$ for  $\overline \cS\left(\lambda,\Lambda,f\right),\underline \cS\left(\lambda,\Lambda,f\right),  $ and $\cS^*\left(\lambda,\Lambda,f\right),$   
 respectively. 
\end{definition}
 
%In the  paper, we will prove  a uniform $W^{2,\ve}$-estimate  of solutions in   $\cS^*(f)$  on Riemannian manifolds with a lower bound of sectional curvature.  

 %%%%%%%%%%%%%%%%%%%%%%%%%%%%%%%%%%%%%%%%%%%%%%%%%%%%%%%%%%%%%%%%%%%%%%%%%%%%%%%%%%%%%%%%%%%%%%%%%%%%%%% section
%%%%%%%%%%%%%%%%%%%%%%%%%%%%%%%%%%%%%%%%%%%%%%%%%%%%%%%%%%%%%%%%%%%%%%%%%%%%%%%%%%%%%%%%%%%%%%%%%%%%%%
\section{Uniform $W^{2,\ve}$-estimate  for elliptic operators}
%%%%%%%%%%%%%%%%%%%%%%%%%%%%%%%%%%%%%%%%%%%%%%%%%%%%%%%%%%%%%%%%%%%%%%%%%%%%%%%%%%%%%%%%%%%%%%%%%%%%%%% subsection
%%%%%%%%%%%%%%%%%%%%%%%%%%%%%%%%%%%%%%%%%%%%%%%%%%%%%%%%%%%%%%%%%%%%%%%%%%%%%%%%%%%%%%%%%%%%%%%%%%%%%%
\subsection{ABP type estimate}%\label{subsec-ABP type estimate}
%%%%%%%%%%%%%%%%%%%%%%%%%%%%%%%%%%%%%%%%%%%%%%%%%%%%%%%%%%%%%%%%%%%%%%%%%%%%%%%%%%%%%%%%%%%%%%%%%%%%%%%  
 We   recall the ABP type estimate on Riemannain manifolds established  by Cabr\'e \cite{Ca} that plays an important role in the proof of the Harnack inequality and $W^{2,\varepsilon}$-estimate for fully nonlinear elliptic operators in nondivergence form.  %As mentioned before, the
 A  major difficulty in proving the ABP estimate on manifolds is that non-constant affine functions can not be generalized to an intrinsic notion on general manifolds. Cabr\'e \cite[Lemma 4.1]{Ca} replaced affine functions by quadratic functions which are   squared distance functions  in order to show  the ABP type estimate. On the basis of this idea, Wang and Zhang \cite{WZ} introduced the contact set defined as follows% (see [Sav] for the Euclidean case) 
% We present  the notations  of contact sets  introduced by Wang and Zhang  \cite{WZ}  defined as follows 
; see   \cite{CC,S,W} for the Euclidean case.  

\begin{definition}[Contact set]\label{def-contact-set}
  Let $\Omega$ be a bounded, open set in $M$ and let $u\in C(\Omega).$ For a given $h>0$ and a compact set $E\subset M,$ the   contact set associated with $u$ of opening $h$ with  vertex set $E$ is defined by
  $$\underline\cG_h(u;E;\Omega):=\left\{x\in\Omega :    \,\inf_{ \Omega}\left(u+\frac{h}{2}d_y^2\right)=u(x)+\frac{h}{2}d_y^2(x)\quad\mbox{for some $y\in E$} \right\}.$$
% We define the upper contact set denoted by $\overline\cG_h$ by   
  %$$\overline\cG_h(u;E;\Omega):=\left\{x\in\Omega : \exists y\in E \,\,s.t.\,\,\sup_{ \Omega}\left(u-\frac{h}{2}d_y^2\right)=u(x)-\frac{h}{2}d_y^2(x) \right\}.$$ 
 %And we define 
 %$$\cG_h(u;E;\Omega):=\underline\cG_h(u;E;\Omega)\cap \overline\cG_h(u;E;\Omega).$$
 \end{definition}

% It is the set of all points $x$ at which a concave paraboloid $-\frac{a}{2}d_y^2(\cdot)+c$ (for some $c\in\R$) of an opening $a$ and with a vertex $y\in E$   touches $u$  from below. In the classical    ABP estimate \cite{CC}, we consider the contact set  of $u$ using   the convex envelop of $u$, especially  affine functions,   which contains a contact point  $x\in\Omega$   satisfying the following equality with   some $p\in\R^n$

\begin{remark}\label{rmk-c-convex}
{\rm
%For an bounded open set $\Omega,$ 
We recall the $c$-convexification of  $u\in C(\Omega)$ for a given  cost function $c(x,y),$ %:=\frac{h}{2}d^2(x,y),$  
defined as 
$$\displaystyle u^{cc}(x):=\sup_{y\in E}\inf_{z\in\Omega} \left\{ u(z)+c(z,y)-c(x,y)\right\}\quad\forall x\in\Omega.$$
It is easy to check     that $u\geq u^{cc}$   and   $u^{cc}$ is continuous in $ \Omega$ if  $c$ is continuous. One can   also prove that  for $c(x,y)=\frac{h}{2}d^2(x,y),$
$$\underline\cG_h\left(u; E;\Omega\right)=\left\{x\in\Omega : u(x)=u^{cc}(x) \right\}.$$ 
%Since  $u^{cc}=u^{cccc},$$ u^{cc}$ has a global  tangent concave paraboloid from below everywhere, which  implies that $u^{cc}$ is semi-convex    in $\Omega.$   
 We refer to \cite[Chapter 5]{V} for more details about $c$-convex functions.    
 }
 \begin{comment}
$$$$

{ \color{black}*} We claim that $$\underline\cG_a(w)=\left\{w=w^{cc}\right\}\quad \mbox{is closed in $\overline\Omega$}$$
since a convexification $w^{cc}$ of $w$ is also continuous for $c=\frac{a}{2}d^2(x,y)$.  In fact,
for  $E=\overline\Omega$ and $w\in C(\overline\Omega),$ 
$$\underline\cG_{a}(w)=\left\{w=w^{cc} \right\},$$
where $$\displaystyle w^{cc}(x)=\sup_{y\in\Omega}\inf_{z\in\Omega} \left\{w(z)+c(z,y)-c(x,y)\right\}.$$ 
We observe  that $w^{cc}$ is semi-convex satisfying $w^{cc}\geq0$ on $B_{7R}\setminus B_{6R}$. 
% where $\Gamma^c_w$ is the $c-$ lower envelope of $w$ over $B_{7R}$ for $c=d^2/2.$ Namely, $$\Gamma^c_w= w^{cc}.$$
For $x\notin \Cut(z_0)\cup\Cut(y_0)\cup\Cut(x_0),$\begin{align*}
w=u+v\geq -\frac{a}{2}d_{y_0}^2+\frac{a}{2}d^2_{y_0}(x_0).
\end{align*}
\end{comment}
\end{remark}

To prove the  ABP type estimate, 
%To proceed with ABP method,    the measure of the image of the gradient mapping 
 Jacobian of the normal map $\phi$ on the contact set below, which corresponds to  the image of the gradient mapping  in the Euclidean space,      was computed explicitly by Cabr\'e. 
 %Under the assumption that $M$ satisfies  nonnegative sectional curvature condition or the first condition of \eqref{eq-cond-Kim}, 
   %  the Jacobian determinant of the exponential map $\exp_x$ for any $x\in M$ is   bounded  from above by $1,$ and hence  it follows that  for $y=\phi(x)= \exp_x a^{-1}\D u(x),$
     %$$\Jac\phi(x)\leq\left|\det D^2\left(a^{-1}u+d_y^2/2\right)(x)\right|.$$   
The following lemma is an   improved estimate by  Wang and Zhang \cite[Theorem 1.2]{WZ}  using   a standard theory of Jacobi fields. 
  \begin{lemma}\label{lem-WZ-jac}
Assume that $\Ric \geq -(n-1)\kappa$ for $ \kappa\geq0.$
%\begin{equation*}%\label{eq-ricci-bdd}
%\Ric \geq -(n-1)\kappa  \quad \mbox{on $M,$} \quad\mbox{for $\kappa\geq0$}.
%\end{equation*} 
Let $u$ be a smooth function in $\Omega\subset M.$ Define the normal map $\phi:\Omega \to M$  as 
$$\phi(x):=\exp_x  \D u(x).$$ 
For a given point $x\in\Omega,$  assume that 
$[0,1]\owns t \mapsto \exp_x t\D u(x)$ is a unique minimizing geodesic joining $x$ to  $\phi(x). $ 
%$ \D u(x)\in E_x\subset T_xM,$ where  $\exp_x : E_x \to \exp_x(E_x)$ is a diffeomorphism. 
 Then we have
 %\begin{enumerate}[(a)]
%\item If $y\in E$  satisfies
  %   $$\displaystyle\min_{\overline \Omega}\left(u+\frac{h}{2}d_y^2\right)=u(x)+\frac{h}{2}d_y^2(x),$$
 %then $$y= \phi(x)=\exp_{x}h^{-1}\D u(x), \quad\mbox{and}  \quad x\notin \Cut(y).$$
%\item
 
\begin{equation*}%\label{eq-jac-est}
\Jac \phi(x)\leq  {\sS}^n\left(\sqrt{ {\kappa}} {|\D u(x)|}\right)\left\{\sH\left(\sqrt{ {\kappa}} {|\D u(x)|}\right)+ \frac{\La u(x)}{n}\right\}^n,
\end{equation*}  
where
$\sH(\tau)=\tau\coth(\tau),$ and $\sS(\tau)=\sinh(\tau)/\tau$ for $\tau\geq0.$
%In particular, if $\Ric\geq0,$ i.e., $\kappa=0,$ we have  $$\Jac\phi(x)\leq    \left( 1+\frac{\La u}{nh}\right)^n.$$
%\end{enumerate}
\end{lemma} 
Using the Jacobian estimate of the normal map in Lemma \ref{lem-WZ-jac}, we have the following  ABP type estimate in       \cite[Lemma 4.1]{Ca} and \cite[Theorem 1.2]{WZ}.  
 \begin{lemma}[ABP type estimate \cite{Ca, WZ}]%t\label{lem-abp}
Assume that
%\begin{equation*}%\label{eq-ricci-bdd}
$\Ric \geq -(n-1)\kappa$ on $M$ for $\kappa\geq0.$ %  \quad \mbox{on $M$} \quad\mbox{for $\kappa\geq0$}.
%\end{equation*} 
For  $z_0,x_0\in M$ and $0<r\leq R,$ assume that $  B_{7r}(z_0)\subset B_{R}(x_0).$   Let $u$ be a smooth function  on $    B_{R}(x_0)$ such that %  $\cM^-(D^2u)\leq f$ in $B_{7r}(z_0),$
\begin{equation*}%\label{eq-abp-cond-u}
u\geq0\quad\mbox{on $B_{R}(x_0)\setminus B_{5r}(z_0) \quad$ and $\quad \displaystyle\inf_{B_{2r}(z_0)}u\leq 1.$}
\end{equation*}
Then we have that 
%$$B_r(z_0)\subset \phi\left( \underline\cG_{r^{-2}}(u)\cap B_{5r}(z_0) \right)\quad\mbox{for $\,\,\phi(x):=\exp_x  r^2\D u(x),$}$$ and 
  %$$\frac{\left|\underline\cG_{r^{-2}}\left(u+v_\delta(\cdot; r;z_0)\right)\cap B_{\delta r}(z_0)\right|}{ |B_{7r}(z_0)|} \geq \mu_\delta>0,$$
\begin{equation*}%\label{eq-area-jac}
\left|B_r(z_0)\right|\leq  \int_{\underline\cG_{r^{-2}}(u)\cap B_{5r}(z_0)} {\sS}^n\left(\sqrt{ {\kappa}} r^2{|\D u|}\right)\left\{\sH\left(\sqrt{ {\kappa}} r^2{|\D u|}\right)+  \frac{r^2\La u}{n}\right\}^n.
\end{equation*}
  where $\underline\cG_{r^{-2}}(u):=\underline\cG_{r^{-2}}\left(u;\overline B_{7r}(z_0); B_{R}(x_0)\right).$
\end{lemma}

 Making use  of  the ABP type estimate  and well-understood barrier functions below, we will
 investigate   in Lemma \ref{lem-abp-type}   the measure   of the contact set at  which the second derivatives  of the  supersolution    have  a uniform lower bound. % from th.   
%In order to obtain  $W^{2,\epsilon}$-estimate of solutions to elliptic equations,   we   make use  of  the ABP type estimate to obtain  estimates of the measure   of contact sets at  which the second derivatives  of the  supersolution  will have  a uniform lower bound in Lemma \ref{lem-abp-type}.   
%Before we state a modified  ABP type estimate in Lemma \ref{lem-abp-type}, 
 First, we recall      from \cite{Ca}  the  barrier function  on Riemannian manifolds and its     
 scale invariant properties.  
\begin{lemma}\label{lem-barrier}
Assume that  $\Sec\geq -\kappa$   for $\kappa\geq0.$ 
For  $z_0,x_0\in M$ and $0< r\leq R\leq R_0,$ assume that $B_{7r}(z_0)\subset B_{R}(x_0).$ For  a given $\delta\in(0,1),$ there exists a continuous function $v_{\delta}=v_\delta(\cdot ;r;z_0)$ in $M$ %  which is smooth in $M\setminus \Cut (z_0),$ 
defined by 
 $$v_\delta(\cdot;r;z_0):=\psi_\delta\left(\frac{d_{z_0}^2}{r^2}\right)$$  
 for a smooth, even    function $\psi_\delta:\R\to\R$ with $\psi_\delta'>0$ on $(0,+\infty),$  and  $v_\delta$ satisfies the following:
\begin{enumerate}[(i)]
 \item $v_{\delta}\geq 0$ in $M\,\setminus\,  B_{5r}(z_0)$, 
\item $v_{\delta}\leq 0$ in $B_{2r}(z_0)$,
\item {  $r^2\cM^+(D^2v_{\delta})+ (n+1)\Lambda\sH(2\sqrt{\kappa}R_0)\leq 0\ $}  in $\left(B_{7r}(z_0)\setminus  B_{ {\delta} r}(z_0)\right)\setminus\Cut(z_0) , $
\item $r^2 D^2 v_{\delta} <  C_\delta {\bf I}$ in ${\color{black}B_{R}(x_0)}\setminus \Cut(z_0),$%$\overline B_{7R}\setminus \Cut(z_0)$,
\item $\displaystyle \frac{1}{R}{r^2|\D v_\delta|}<C_\delta$ in ${\color{black}B_{R}(x_0)}\setminus \Cut(z_0),$% {\color{red} NEED?}
\item  $v_{\delta}\geq -C_{\delta}$ in $M,$
\end{enumerate}
where  $\sH(t):=t\coth(t)$ for  $t\geq 0. $  
%  with $\sH(0)=0$ and $\sS(0):=0, $
Here, the constant $C_{\delta}>0$  depends only on $\delta, n,\lambda,\Lambda,$ and $\sH(2\sqrt{\kappa}R_0 )$ (independent of  $r$ and $z_0$).
\end{lemma}
%$${\color{blue}\mbox{ Take $\psi$ to be linear at infinity }}$$
\begin{proof}
We give a sketch of the proof; see  \cite[Lemma 5.5]{Ca} and  \cite[Lemma 4.2]{WZ} for details. 
Fix $0<\delta<1$ and 
consider   
  $$\psi_\delta(s):=\left(\frac{3^2}{5^2}\right)^{-\al}-  \left(\frac{s}{5^2}\right)^{-\al}\quad\mbox{for $s\geq\delta^2$, }$$  for a   positive constant $ \alpha$ to be chosen later, which will depend only on $\delta, n,\lambda,\Lambda, $ and $\sH(2\sqrt{\kappa}R_0).$ 
    After fixing a large contant $\al>0,$ we will extend $\psi_\delta$ smoothly   in $\R$ to be an even function and to satisfy  that $\psi'>0$ in $(0,+\infty).$    % and that $\psi_\delta$ is linear on $[10,+\infty).$ 
     Now we define
$$v_\delta(\cdot;r;z_0) :=\psi_\delta\left(\frac{d^2_{z_0}}{r^2}\right) \quad\mbox{in}\,\,\,\, M, $$ where $d_{z_0}$ is the  distance function to $z_0.$ %We may assume that $\psi_\delta(s)$ is increasing with respect to $s.$   
 It is clear that (i), (vi) and (ii) hold for  $\al\geq 1.$

In order to check (iii), we  recall  that a closed set  $\Cut(z_0)$ has measure zero and that 
%\begin{align*}
%\left\langle D^2\left(d_{z_0}^2/2\right)(x)\cdot \xi,\xi\right\rangle=\left\langle d_{z_0}D^2d_{z_0}(x)\cdot \xi,\xi\right\rangle+\left\langle\D d_{z_0}(x),\xi\right\rangle^2,\quad\forall\xi\in T_xM,\,\,\, x\not\in \Cut(z_0),\end{align*}
%and  
$$\cM^+\left(D^2d_{z_0}^2(x)\right)\leq 2n\Lambda \sH\left(\sqrt\kappa d_{z_0}(x)\right)\leq n\Lambda \sH\left( 2\sqrt{\kappa}R_0\right)\quad\forall  x\in B_{7r}(z_0)\setminus\Cut(z_0),
$$
from Lemma \ref{lem-hess-dist-sqrd}.  %and  following  the proof  of , %and using Lemma \ref{lem-prop-pucci-op} $(a),$    
As in the proof of  \cite[Lemma 5.5]{Ca} and  \cite[Lemma 4.2]{WZ},  we can select      $ \al>0$  sufficiently  large so  that  (ii), (iii) hold, where  $\al>0$ depends only on $\delta, n,\lambda,\Lambda, $ and $\sH(2\sqrt{\kappa}R_0)$. Lastly, 
for a fixed $\al>0,$ it is not hard to  check   (iv) and  (v)    using Corollary \ref{cor-hess-dist-sqrd} since 
$\psi_\delta'$ is bounded in $[0,+\infty),$ and $\psi_\delta''(s)\leq 0$ for $s\geq 1.$
\end{proof}

\begin{remark} \label{rmk-barrier}
{\rm
In Lemma \ref{lem-barrier} (iv), we have obtained an upper bound of $r^2 D^2v_\delta(\cdot;r;z_0)$ in $B_{R}(x_0)\setminus\Cut(z_0). $ 
 According to Corollary \ref{cor-hess-dist-sqrd}, we deduce that   for any $x\in B_{R}(x_0)$ and any unit vector  $X\in T_{x}M,$$$\limsup_{t\to0} \frac{1}{t^2}\left\{v_\delta\left(\exp_{x}tX;r;z_0\right)+v_\delta\left(\exp_{x}-tX; r;z_0\right)-2v_\delta(x;r;z_0)\right\}\leq \frac{C_\delta}{ r^2},$$
 where $C_\delta$ is the same constant as in Lemma \ref{lem-barrier}. 
}
\end{remark}

%Let $v$ be a barrier function  defined as$$v :=\psi\left(\frac{d^2_{z_0}}{R^2}\right) $$ for an increasing  function $\psi:\R_+\to\R$ with $\psi'>0$ on $(0,+\infty).$ Note that $\psi\in C^2$ and $ v$ is smooth on $B_{7R}\setminus \Cut(z_0).$

%In order to estimate the distribution function of $|D^2u|,$ we will modify  the ABP type estimate and use the Calderon-Zygmund technique used in \cite{Ca}.  
Since  the  barrier function in Lemma \ref{lem-barrier} are not smooth on the cut locus of the center   point, we need the following technical lemma  to apply the ABP type estimate to a sum of a smooth function and  the scale invariant barrier function directly  in Lemma \ref{lem-abp-type}. 
%to obtain an improved   ABP type estimate (Lemma \ref{lem-abp-type}). 

 \begin{lemma}\label{lem-contact-set-tech}
 Assume that  $\Sec\geq -\kappa$   for $\kappa\geq0.$ For    $z_0,x_0\in M$ and $0<r\leq R\leq  R_0,$  assume that $B_{7r}(z_0)\subset B_{R}(x_0).$
%Let $M$ be a complete Riemannian manifold with a lower bound $-\kappa$ ($\kappa>0$) of sectional curvature. 
 Let $u$ be a smooth function on $  B_{R}(x_0),$ % and $v(\cdot;R)$ be a barrier function in Lemma \ref{lem-barrier}. 
 and let $$w:=u+h\sum_{j=1}^{k}r_j^2v_\delta(\cdot;r_j;z_j) +h\sum_{j=1}^{k}\frac{1}{2} d_{y_j}^2+v_ \delta(\cdot;r;z_0)$$
 for $k\in\N\cup\{0\},$ where $h>0,$ $r_j>0,$     $ B_{7r_j}(z_j)\subset B_{R}(x_0),  y_j\in \overline B_{7r_{j}}(z_j)$ for   $j=1,\cdots,k,$ and $v_ \delta(\cdot; r;z)$ is the barrier function with respect to $r$ and $z$ as in Lemma \ref{lem-barrier}.     
 Assume that $x\in \underline\cG_{r^{-2}}\left(w; \overline B_{7r}(z_0);B_{R}(x_0) \right),$ that is, for some  $y_0\in\overline B_{7r}(z_0),$ 
 $$ \inf_{B_{R}(x_0)}\left(w+\frac{1}{2r^2}d_{y_0}^2\right)=w(x)+\frac{1}{2r^2}d_{y_0}^2(x).$$
 Then we have the following:
%  $w$ is  touched by $-\frac{1}{2R^2}d_{y_0}^2$ at $x_0\in B_{7R}(z_0)$  from below up to translation for some $y_0\in \overline B_{7R}(z_0).$ Then    
 \begin{enumerate}[(i)]
 \item $\displaystyle  x\not\in \bigcup_{j=0}^{k}\Cut(z_j)\cup \bigcup_{j=0}^{k}\Cut(y_j)$ 
\item $w$ is smooth at  $x, $ and satisfies that $r^2|\D w(x)|= d_{y_0}(x)<2R$    and 
 \begin{align*}
       -\sH\left(2\sqrt{\kappa}R_0\right){\bf{I}}& \leq  r^2 D^2w(x)\leq  r^2D^2u(x)+  r^2 hk\left\{C_\delta+\sH\left(2\sqrt{\kappa}R_0\right)\right\}{\bf{I}} +C_\delta{\bf I}%+ \sH\left(2\sqrt{\kappa}R_0\right){\bf{I}},
    %   K_1 \left( \max _{1\leq j\leq K_1}a_j\right) C_\delta(\sqrt{\kappa}R_0){\bf{I}}+K_2 \left( \max _{1\leq j\leq K_2}b_j\right)\sH\left(2\sqrt{\kappa}R_0\right){\bf{I}},
       \end{align*}
 where $C_\delta>0$  %=C_\delta\left(\delta,n,\lambda,\Lambda,sH\left(2\sqrt{\kappa}R_0\right)\right)>0$  
 is the constant as in Lemma \ref{lem-barrier}
 \end{enumerate}
 \end{lemma} 
 \begin{proof} 
 Once (i) is proved, (ii) easily follows from     Lemma \ref{lem-hess-dist-sqrd} and Lemma \ref{lem-barrier}. Note that $\sH$ is nondecreasing in $[0,+\infty).$  So it suffices to show that $\displaystyle  x\not\in \bigcup_{j=0}^{k}\Cut(z_j)\cup \bigcup_{j=0}^{k}\Cut(y_j).$
 % First, we show that $x\not\in \bigcup_{j=0}^l\Cut(z_j).$
 We will only prove that $x\not\in\Cut(z_0)$ since the  proofs for the other cases    are      similar. 

 Suppose to the contrary that $x\in\Cut(z_0).$ % for some $k\in\{0,1,\cdots,l\}.$ 
% We note that $x_0\not\eq z_0.$ 
Since $w-w(x)$ lies above  $-\frac{1}{2r^2}d_{y_0}^2+\frac{1}{2r^2}d_{y_0}^2(x)$,  we take the second difference quotient and  use Lemma \ref{lem-hess-dist-sqrd} 
  to have that for  any unit vector $X\in T_{x}M$,  %with $|X|=1,$
 \begin{equation*}%\label{eq-2nd-diff-quo}
 \begin{split}
&\left\langle r^2 D^2u\cdot X,X\right\rangle_{x}+r^2\liminf_{t\to0} \frac{1}{t^2}\left\{v_\delta \left(\exp_{x}tX;r;z_0\right)+v_\delta  \left(\exp_{x}-tX; r;z_0 \right)-2v_\delta  \left(x;r;z_0 \right)\right\}\\
&+hr^2\sum_{j=1}^{k}r_j^2\limsup_{t\to0} \frac{1}{t^2}\left\{v_\delta \left(\exp_{x}tX;r_j;z_j\right)+v_\delta  \left(\exp_{x}-tX; r_j;z_j \right)-2v_\delta  \left(x;r_j;z_j \right)\right\} \\
&+hr^2\sum_{j=1}^{k}\limsup_{t\to0}  \frac{1}{2t^2}\left\{d_{y_j}^2(\exp_{x}tX)+d_{y_j}^2(\exp_{x}-tX)-2d_{y_j}^2(x)\right\}\\
&\geq - \limsup_{t\to0}\frac{1}{2t^2}\left\{d^2_{y_0}(\exp_{x}tX)+ d^2_{y_0}(\exp_{x}-tX)-2d^2_{y_0}(x)\right\}\\
&\geq -\sH\left(\sqrt{\kappa}d(x,y_0)\right)\geq \sH\left(2\sqrt{\kappa}R_0\right).
\end{split}
 \end{equation*}
 Combined with  Remark \ref{rmk-barrier} and Lemma \ref{lem-hess-dist-sqrd} , this implies that for any unit vector  $X\in T_xM$, %for each $j=0,\cdots,l$
 $$r^2\liminf_{t\to0} \frac{1}{t^2}\left\{v_\delta  \left(\exp_{x}tX;r;z_0 \right)+v_\delta  \left(\exp_{x}-tX; r;z_0 \right)-2v_\delta  \left(x;r;z_0 \right)\right\}\geq -C>-\infty.$$
 %where we set $r_0:=R.$   % Since $x\not\eq z_k,$ we use
  Using  strict monotonicity  of $\psi_\delta $ in Lemma \ref{lem-barrier},  we find  a positive constant $c_0>0$ such   that for   small  $|t|\in(0,1)$ % continuity of $d_{z_0}$
  \begin{align*}
  & \frac{v_\delta  \left(\exp_{x}tX;r;z_0 \right)-v_\delta  \left(x;r;z_0 \right)}{d_{z_0}^2(\exp_{x}tX)-d_{z_0}^2(x)}\leq 2r^2 {\psi'\left(\frac{d_{z_0}^2(x)}{r^2}\right)}=:c_0,
  \end{align*}
   where we notice that $x\not\eq z_0$ and hence  $\psi'\left(d_{z_0}^2(x)/r^2 \right)$  is positive  since we assume  $x\in\Cut(z_0).$  
Thus we   deduce that   for any unit vector  $X\in T_xM,$ % with $|X|=1,$ 
\begin{align*}
&\liminf_{t\to0} \frac{d^2_{z_0}(\exp_{x}tX)+ d^2_{z_0}(\exp_{x}-tX)-2d^2_{z_0}(x)}{t^2}\\
&\geq \frac{1}{c_0}\liminf_{t\to0}\frac{ v_\delta \left(\exp_{x}tX;r;z_0 \right)+ v_\delta \left(\exp_{x}-tX;r;z_0 \right)-2v_\delta\left(x;r;z_0 \right)}{t^2}\geq -\frac{C}{c_0r^2}>-\infty,
\end{align*}
which contradicts to the assumption that $x \in\Cut(z_0)$ from Lemma \ref{lem-dist-sqrd-cut}.
 Therefore,  $x$ is not a cut point of $z_0,$ which finishes the proof.  
 \begin{comment}

 Similarly, we can show that   $x$ is not a cut point of  $z_k$ for any $k\in\{0,1,\cdots,l\}.$  

Now we prove that $x$ is not a cut of  $y_j$ for all $j=0,1,2,\cdots, l.$ In a similar way,  \eqref{eq-2nd-diff-quo} and  \cite[Proposition 2.5]{CMS}, it is easy to see that 
$x_0\not\in\bigcup_{j=1}^l\Cut(y_j)$
since 
$$\liminf_{t\to0}  \frac{1}{2t^2}\left\{d_{y_j}^2(\exp_{x_0}tX)+d_{y_j}^2(\exp_{x_0}-tX)-2d_{y_j}^2(x_0)\right\}>-\infty \quad\forall j=1,\cdots,l.$$
% which implies that  w is twice differentiable $x_0.$  
It remains to show that $x_0$ is not a cut point of $y_0.$ From the assumption, $\frac{1}{2}d_{y_0}^2-\frac{1}{2}d_{y_0}^2(x_0)$ lies above  and equals to $-R^2w+R^2w(x_0)$ at $x_0$ in $B_{R_0}(x_0).$ We notice that $R^2w$ is smooth in a small neighborhood of $x_0$ since $x\not\in \bigcup_{j=0}^{K_1}\Cut(z_j)\cup \bigcup_{j=1}^{K_2}\Cut(y_j).$
 Thus we have  that 
$$\liminf_{t\to0}  \frac{1}{2t^2}\left\{d_{y_0}^2(\exp_{x_0}tX)+d_{y_0}^2(\exp_{x_0}-tX)-2d_{y_0}^2(x_0)\right\} \geq -R^2 \left\langle R^2 D^2w\cdot X,X\right\rangle_{x_0}>-\infty,$$
which means that $x_0$ is not a cut point of $\Cut(y_0)$    once again from  \cite[Proposition 2.5]{CMS}. 

\end{comment}
\end{proof}

 In the following,  we obtain  the   measure estimate  of  the   contact set that consists of   points, where   $u+v_\delta $ has a global tangent concave paraboloid from below.    %has information about the second  
      \cite[Lemma 5.1]{Ca}, \cite[Lemma 3.1]{K} and   \cite[Proposition 4.1]{WZ}   dealt with 
  estimates of the measure of the  level sets of the solution $u$ to establish    
    pointwise estimates; the Harnack  estimate and the weak Harnack inequality.   % starting from estimates of the measure of its level sets. 
  %   Instead,  
  %, which enables  us to  make  iteration procedure  for 
%  In order to   estimate  the distribution function of $|D^2u|,$ the norm of  the Hessian,   we keep the barrier function in  the estimate of   Lemma \ref{lem-abp-type} as follows. 
  In order to  study the  bound of the second derivatives of the solution, % in $L^p$ space,  
  we keep the barrier function in  the estimate of   Lemma \ref{lem-abp-type}.  This  is the first step to  estimate  the distribution function of $|D^2u|,$ the norm of  the Hessian. 
 
 % We    provide a detailed proof for  the   ABP type estimate in Lemma \ref{lem-abp-type} which is a   modified version of \cite[Lemma 4.1, Lemma 5.1]{Ca} and   \cite[Theorem 1.2, Proposition 4.1]{WZ}. 

% As mentioned before, the major difficulty in proving the ABP estimate on manifolds is that non-constant affine functions can not be generalized to an intrinsic notion on general manifolds. 
 \begin{lemma}\label{lem-abp-type}
   Assume that  $\Sec\geq -\kappa$   for $\kappa\geq0.$ 
For  $z_0,x_0\in M$ and $0<r\leq R\leq R_0,$ assume that $  B_{7r}(z_0)\subset B_{R}(x_0).$   Let $u$ be a smooth function  on $ \overline  B_{R}(x_0)$ such that  $u\in\overline \cS\left(\lambda,\Lambda, f\right)$
%$\cM^-(D^2u)\leq f$
 in $B_{7r}(z_0),$
\begin{equation}\label{eq-abp-cond-u}
u\geq0\quad\mbox{on $B_{R}(x_0)\setminus B_{5r}(z_0) \quad$ and $\quad \displaystyle\inf_{B_{2r}(z_0)}u\leq 1.$}
\end{equation}
For a given $\delta\in(0,1),$ there exist   uniform constants $\e_\delta>0$ and $\mu_\delta>0$ depending only on $\delta,n, \lambda,\Lambda$ and $\sqrt{\kappa}R_0,$ such that if
  $$\left(\frac{1}{|B_{7r}(z_0)|}\int_{\underline\cG_{r^{-2}}\left(u+v_\delta(\cdot;r;z_0)\right)\cap B_{7r}(z_0)} |r^2 f^+|^{n\eta}\right)^{\frac{1}{n{\color{black}\eta}}}\leq \e_\delta;\quad\eta:=1+\log_2\cosh (4\sqrt\kappa R_0),$$
  then 
  $$\frac{\left|\underline\cG_{r^{-2}}\left(u+v_\delta(\cdot; r;z_0)\right)\cap B_{\delta r}(z_0)\right|}{ |B_{7r}(z_0)|} \geq \mu_\delta>0,$$
  where $\underline\cG_{r^{-2}}(u+v_\delta(\cdot; r;z_0)):=\underline\cG_{r^{-2}}\left(u+v_\delta(\cdot; r;z_0);\overline B_{7r}(z_0); B_{R}(x_0)\right)$ and $v_\delta(\cdot;r;z_0)$ is as in Lemma \ref{lem-barrier}.  
  \end{lemma}
  \begin{proof}
Let  $\Omega:= B_{R}(x_0),$  $E:=\overline B_{7r}(z_0),$ and
 $$w:= u+ v_\delta(\cdot; r;z_0).$$   
 For   $x\in \underline\cG_{r^{-2}}(w):=\underline\cG_{r^{-2}}\left(w;E;\Omega\right),$
there exists $y\in E$ such that  
\begin{equation*}
\inf_{ B_{R}(x_0)}\left(w+\frac{1}{2r^2}d_y^2\right)=w(x)+\frac{1}{2r^2}d_y^2(x).
\end{equation*}
According to Lemma \ref{lem-contact-set-tech},  % and \eqref{eq-dist-differential}, 
one can check  that  $x\not\in\Cut(z_0)\cup\Cut(y), $
\begin{equation*}%\label{eq-properties-contact-pt}
\begin{split}
&\D w(x)=-r^{-2}d(x,y)\D d_y(x),\quad D^2\left(w+\frac{1}{2r^2}d_y^2\right)(x)\geq0,\quad\mbox{and}\quad\\
 %CMS 
&y=\exp_x r^{2}\D w(x)\not\in\Cut(x).
\end{split}
\end{equation*} %Ca  
%where we  refer the proof of \cite[Lemma 4.1]{Ca} for the proof of the last equality. 

 Now, consider the smooth function $\phi:B_{R}(x_0)\setminus\Cut(z_0)\to M$  defined by
 $$\phi(x):=\exp_xr^{2}\D w(x).$$ 
  %which is well defined. 
% We observe that $\Omega_1\subset B_{7R}\setminus \Cut(z_0)$ is an open set and  $w_\delta:=u+v_\delta$ is smooth in $\Omega_1$  
   We use   \eqref{eq-abp-cond-u} and the properties (i),(ii) of $v_\delta(\cdot;r;z_0)$ in Lemma \ref{lem-barrier} to deduce   that 
  $$B_r(z_0)\subset \phi\left(\underline\cG_{r^{-2}}(w)\cap B_{5r}(z_0)\right),$$
where we refer to  the proof of  \cite[Lemma 4.1]{Ca} for   details.
 % and 
  %$$\underline\cG_{R^{-2}}(w_\delta) \subset  \Omega_0:= \left\{ z\in \Omega_1:  |\D w_\delta(z)|< 14/R \right\}.$$
From  Remark \ref{rmk-c-convex}  and  Lemma \ref{lem-contact-set-tech}, we
  observe that  $\underline\cG_{r^{-2}}(w)$ is  measurable and 
  %depending   on $\delta, n,\lambda,\Lambda,$ and $\sqrt{\kappa}R$ satisfying
 $$\underline\cG_{r^{-2}}(w)\subset
\left\{x\in B_{R}(x_0)\setminus \Cut(z_0) :   -C {\bf I}< r^2 D^2w(x)<C{\bf I},\,\,\,r^2|\D w(x)|< 2R_0  \right\}=:\Omega_0 $$
for some $C>0,$ where  $\phi$ is Lipschitz in a bounded, open set  $\Omega_0. $ Note that a closed set  $ \Cut(z_0)$  has measure zero. 
  Now we apply the area formula to obtain
  %=\phi\left(\left\{w=w^{cc}\right\}\right).$$
 \begin{equation}\label{eq-area-jac}
\left|B_r(z_0)\right|\leq\left|\phi\left(\underline\cG_{r^{-2}}(w)\cap B_{5r}(z_0)\right)\right|\leq \int_{\underline\cG_{r^{-2}}(w)\cap B_{5r}(z_0)} \Jac\phi(x)\, d\vol(x).
\end{equation}
%In order   to compute the Jacobian of $\phi$ in $\underline\cG_{R^{-2}}(w),$
%we  make use of   \cite[Theorem 1.2]{WZ}  and hence we have  
\begin{comment}
It follows from Lemma \ref{lem-WZ-jac} that % \cite[Theorem 1.2]{WZ} that 
for $x\in\underline\cG_{r^{-2}}(w),$
\begin{align*}
\Jac\phi(x)\leq \sS^n\left(\sqrt{\kappa}r^{2}|\D w(x)|\right)\left\{\sH\left(\sqrt{\kappa} r^2|\D w(x)|\right)+\frac{r^2\La w(x)}{n}\right\}^n,
\end{align*}
where   $\sH(\tau):=\tau\coth(\tau),$ and $\sS(\tau):= \sinh(\tau)/\tau$ for   $\tau\geq 0.$ 
\end{comment}
%  with $\sH(0)=0$ and $\sS(0):=0. $ 
By making use of %\eqref{eq-properties-contact-pt}  and 
  Lemma \ref{lem-WZ-jac} and    Lemma \ref{lem-contact-set-tech}, we  have that for $x\in\underline\cG_{r^{-2}}(w),$
\begin{equation}\label{eq-jac-est}
\begin{split}
\Jac\phi(x)&\leq \sS^n\left(\sqrt{\kappa}r^{2}|\D w|\right)\left\{\sH\left(\sqrt{\kappa} r^2|\D w|\right)+\frac{r^2\La w}{n}\right\}^n\\
&\leq \sS^n\left(2\sqrt{\kappa}R_0\right)\left[\sH\left(2\sqrt{\kappa}R_0\right)+\frac{1}{\lambda }\left\{\cM^-\left(r^2D^2w\right)+{n\Lambda}\sH\left(2\sqrt{\kappa}R_0\right)\right\}\right]^n\\
&\leq \sS^n\left(2\sqrt{\kappa}R_0\right)\left\{ \frac{r^2 }{\lambda } \cM^-\left(D^2w\right)+(n+1)\frac{\Lambda}{\lambda}\sH\left(2\sqrt{\kappa}R_0\right)\right\}^n\\
&\leq \frac{\sS^n\left(2\sqrt{\kappa}R_0\right)}{\lambda^n}\left\{  r^2f^++r^2\cM^+\left(D^2v_\delta(\cdot; r;z_0)\right)+(n+1) {\Lambda}\sH\left(2\sqrt{\kappa}R_0\right)\right\}^n,
%&\leq \sS^n\left(14\sqrt{\kappa}R\right)\left\{ \frac{R^2 }{\lambda } \cM^-(D^2u)+\frac{R^2}{\lambda}\cM^+(D^2v_\delta)+2(n+1)\frac{\Lambda}{\lambda}\sH\left(14\sqrt{\kappa}R\right)\right\}^n
\end{split}
\end{equation}
where we recall that   $\sS(\tau)$ and  $\sH(\tau)$ are nondecreasing for $\tau\geq0.$   
  From Lemma \ref{lem-barrier}, we notice    that  
 $$r^2\cM^+\left(D^2v_\delta (\cdot; r;z_0)\right)+(n+1) {\Lambda}\sH\left(2\sqrt{\kappa}R_0\right) \leq \left\{ n\Lambda C_\delta + (n+1) {\Lambda}\sH\left(2\sqrt{\kappa}R_0\right)\right\}  \,\chi_{B_{\delta r}(z_0)} $$ % \quad\mbox{in $B_{7r}(z_0)\setminus\Cut(z_0)$}$$
 in $ B_{7r}(z_0)\setminus\Cut(z_0)$   for  $C_\delta>0$ as in Lemma \ref{lem-barrier}, where $\chi_{B_{\delta r}(z_0)}$ stands for the characteristic function. 
Combined  with  \eqref{eq-area-jac} and  \eqref{eq-jac-est}, this %property of $v_\delta (\cdot; r;z_0) $ 
provides that 
 \begin{align*}
\left(\frac{\left|B_r(z_0)\right|}{\left|B_{7r}(z_0)\right|}\right)^{\frac{1}{n}} &\leq\left(\frac{1}{\left|B_{7r}(z_0)\right|} \int_{\underline\cG_{r^{-2}}(w)\cap B_{7r}(z_0)} \Jac\phi(x)\, d\vol(x)\right)^{\frac{1}{n}}\\
&\leq \tilde C \left( \frac{1}{|B_{7r}(z_0)|}\int_{\underline\cG_{r^{-2}}(w)\cap B_{7r}(z_0)} \left|r^2f^+\right|^{n\eta}\right)^{\frac{1}{n\eta}}
+\tilde C\frac{  \left| \underline\cG_{r^{-2}}(w)\cap B_{\delta r}(z_0)\right| ^\frac{1}{n\eta}}{{\left|B_{7r}(z_0)\right|^{\frac{1}{n\eta}} } },
\end{align*}
for a uniform constant $\tilde C>0$ depending only on $\delta,n,\lambda,\Lambda,$ and $\sqrt{\kappa}R_0.$ Using Bishop-Gromov's Theorem \ref{thm-BG},   we have that 
 \begin{align*}%\label{eq-abp-type}
\frac{\left| \underline\cG_{r^{-2}}(w)\cap B_{\delta r}(z_0) \right|^\frac{1}{n\eta}}{{\left|B_{7r}(z_0)\right|^{\frac{1}{n\eta}} } }
+\left(  \frac{1}{|B_{7r}(z_0)|}\int_{\underline\cG_{r^{-2}}(w)\cap B_{7r}(z_0)} \left|r^2f^+\right|^{n\eta}\right)^{\frac{1}{n\eta}}
&\geq \frac{1}{\tilde C} \left\{ \frac{1}{\cD}\left(\frac{1}{7}\right)^{\log_2\cD}\right\}^{\frac{1}{n}}=: 2\mu_\delta^\frac{1}{n\eta}
 \end{align*}
for {\color{black}$\cD:=2^{n}\cosh^{n-1}(2\sqrt\kappa R_0).$}  %and   for a uniform constant $C_3>0$  depending only on $\delta, n,\lambda,\Lambda$ and $\sqrt{\kappa}R_0.$
 By selecting $\e_\delta:=\mu_\delta^{\frac{1}{n\eta}},$ we conclude that 
\begin{align*}%\label{eq-abp-type}
  \frac{\left| \underline\cG_{r^{-2}}\left(u+v_\delta(\cdot;r;z_0)\right)\cap B_{\delta r}(z_0) \right|}{\left|B_{7r}(z_0)\right| } \geq\mu_\delta. 
 \end{align*} 
 \end{proof}

\begin{cor}\label{cor-abp-type}
   Assume that  $\Sec\geq -\kappa$   for $\kappa\geq0.$ 
For  $z_0,x_0\in M$ and $0<r\leq R\leq R_0,$ assume that $  B_{7r}(z_0)\subset B_{R}(x_0).$   Let $u$ be   a smooth function  on $ \overline B_{R}(x_0)$ and let  %$P$ be a function in $\cT_l(B_{R_0}(x_0))$ for some $l\in\N,$ i.e.,
$$\tilde u:= u+ h\sum_{j=1}^{k}r_j^2v_\delta (\cdot;r_j;z_j) +h\sum_{j=1}^k\frac{1}{2} d_{y_j}^2,$$
%$$P:= \sum_{j=1}^{l}r_j^2v(\cdot;r_j;z_j) +\sum_{j=1}^l\frac{1}{2} d_{y_j}^2,$$
for   $k\in\N,$ where   $h>0,$  $r_j>0,$   $ B_{7r_j}(z_j)\subset B_{R}(x_0),  y_j\in \overline B_{7r_{j}}(z_j)$ for   $j=1\cdots,k.$  % $ {\color{black}  h_j\in\{0,1\}}.$  
% Let $$\tilde u:= C(u+ P),$$
%for some $C>0,$
  Assume that 
\begin{equation*}%\label{eq-abp-cond-u}
\tilde u\geq0\quad\mbox{on $B_{R}(x_0)\setminus B_{5r}(z_0),  \qquad \displaystyle\inf_{B_{2r}(z_0)}\tilde u\leq 1,$}
\end{equation*}
 and $\tilde u$ satisfies $\displaystyle \cM^-(D^2\tilde u)\leq \tilde f$ a.e. in $B_{7r}(z_0)\setminus\left(\bigcup_{j=1}^k\Cut(z_j)\cup\bigcup_{j=1}^k\Cut(y_j)\right)$ with    
  $$\left(\fint_{B_{7r}(z_0)} |r^2 \tilde f^+|^{n\eta}\right)^{\frac{1}{n{\color{black}\eta}}}\leq \e_\delta.$$
Then  we have 
$$\frac{\left|\underline\cG_{r^{-2}}\left(\tilde u+v_\delta(\cdot; r;z_0)\right)\cap B_{\delta r}(z_0)\right|}{ |B_{7r}(z_0)|} \geq \mu_\delta>0,$$
  where $\underline\cG_{r^{-2}}\left(\tilde u+v_\delta(\cdot; r;z_0)\right):=\underline\cG_{r^{-2}}\left(\tilde u+v_\delta(\cdot; r;z_0);\overline B_{7r}(z_0); B_{R}(x_0)\right)$ % for  $v_\delta(\cdot;r;z_0)$   in Lemma \ref{lem-barrier} 
  and the uniform constants $\e_\delta$, $\mu_\delta >0$ are as in Lemma \ref{lem-abp-type}. 
%$$\frac{|\underline\cG_{R^{-2}}(\tilde u+v_\delta(\cdot; R;z_0))\cap B_{\delta R}(z_0)|}{ |B_{7R}(z_0)|} \geq \mu_\delta>0,$$
  %where $\underline\cG_{R^{-2}}(\tilde u+v_\delta):=\underline\cG_{R^{-2}}\left(\tilde u+v_\delta;\overline B_{7R}(z_0); B_{R_0}(x_0)\right)$ and $v_\delta=v_\delta(\cdot;R;z_0)$ is as in Lemma \ref{lem-barrier}.  
\end{cor}
\begin{proof} Let $\tilde w:=\tilde u+v_\delta(\cdot;r;z_0). $ 
From  Lemma \ref{lem-contact-set-tech},  we
  observe that  
   \begin{align*}
\underline\cG_{r^{-2}}\left(\tilde w\right):=\underline\cG_{r^{-2}}\left(\tilde w; \overline B_{7r}(z_0); B_R(x_0) \right)\subset\tilde\Omega_0,
\end{align*}
where $$\tilde\Omega_0:=
\left\{x\in B_{R}(x_0)\setminus \left(\bigcup_{j=0}^k\Cut(z_j)\cup\bigcup_{j=1}^k\Cut(y_j)\right) :   -C {\bf I}< r^2D^2\tilde w(x)<C{\bf I},\,\,\,r^2|\D \tilde w(x)|< 2R_0  \right\}
$$ for some $C>0. $ We note that    $\underline\cG_{r^{-2}}\left(\tilde w \right)$ is  measurable according to  Remark \ref{rmk-c-convex} 
  and  that a closed set  $ \bigcup_{j=0}^{k}\Cut(z_j)\cup  \bigcup_{j=1}^{k}\Cut(y_j)$  has measure zero. 
Thus $\tilde w $ is smooth in  $\tilde\Omega_0$ and the function 
$\tilde \phi:\tilde\Omega_0\to M$ defined as 
$$\tilde\phi(x):=\exp_xr^2\D \tilde w(x),$$ 
is Lipschitz continuous in $\tilde \Omega_0.$   
  % where  $\phi$ is Lipschitz in $\Omega_0. $ 
  The remaining part of  the proof is the same as the proof of Lemma \ref{lem-abp-type}.
\end{proof}

%When $Q:=Q^{k,\alpha}$,  we denote the predecessor $Q^{k-1,\beta}$ by $\widetilde Q$ for simplicity.

 %$$$$
 \begin{comment}
   With this choice of $\delta,$  Proposition \ref{prop-decay-est-ellip-1st}   can be  applied to rescaled functions $u/M_\delta^k$   at every scale of space by   using  Lemma \ref{lem-CZ} in order     to deduce the power decay of the distribution function of  supersolutions as follows. 
% \begin{lemma}\label{lem-decay-est-ellip}
 Under the same assumption as Proposition \ref{prop-decay-est-ellip-1st},  let $Q_1$ be   a dyadic cube of  generation $k_R $  close enough to $z_0.$ Then there exist uniform constants $0<\e_0,\mu_0<1$ and  $M_0>1$ such that if
     $$\left(\fint_{B_{7R}(z_0)}|R^2f^+|^{n\eta}\right)^{\frac{1}{n\eta}} \leq \e_0,$$ we have 
 \begin{equation}\label{eq-decay-est-ellip}
 \frac{|\{u>M_0^k\}\cap Q_1|}{|Q_1|}\leq (1-\mu_0)^k
 \end{equation}  
 for $k=1,2,\cdots,$ where the constants $\e_0,\mu_0$ and $M_0$  depend only on $n,\lambda,\Lambda, \sqrt{\kappa}R_0.$  It follows from \eqref{eq-decay-est-ellip}  that 
  \begin{equation}\label{eq-decay-est-ellip-cont}
 \frac{|\{u>h\}\cap Q_1|}{|Q_1|}\leq  C_0h^{-\ve_0}\quad \forall h>0,
 \end{equation}  
where $C_0$ and $\ve_0$ are positive constants depending only on $n,\lambda,\Lambda$ and $\sqrt{\kappa}R_0.$

\end{comment}

 %%%%%%%%%%%%%%%%%%%%%%%%%%%%%%%%%%%%%%%%%%%%%%%%%%%%%%%%%%%%%%%%%%%%%%%%%%%%%%%%%%%%%%%%%%%%%%%%%%%%%%% section
%%%%%%%%%%%%%%%%%%%%%%%%%%%%%%%%%%%%%%%%%%%%%%%%%%%%%%%%%%%%%%%%%%%%%%%%%%%%%%%%%%%%%%%%%%%%%%%%%%%%%%
 \subsection{Calder\'on-Zygmund Technique } 
 %%%%%%%%%%%%%%%%%%%%%%%%%%%%%%%%%%%%%%%%%%%%%%%%%%%%%%%%%%%%%%%%%%%%%%%%%%%%%%%%%%%%%%%%%%%%%%%%%%%%%%% section
%%%%%%%%%%%%%%%%%%%%%%%%%%%%%%%%%%%%%%%%%%%%%%%%%%%%%%%%%%%%%%%%%%%%%%%%%%%%%%%%%%%%%%%%%%%%%%%%%%%%%%
We quote this subsection from \cite[Section 6]{Ca}    to introduce  the Calder\'on-Zygmund techinque which is one of main tools for the proof of  uniform $L^p$-estimates of the Hessian of solutions to uniformly  elliptic equations.    
%Adapting the methods of the Calder—nÐZygmund theory %to deduce a power decay  for  the distribution function  of  the Hessian of solutions to uniformly  elliptic equations.   
 We first  present  the Christ decomposition     \cite{Ch},  which generalizes the  Euclidean dyadic decomposition for 
  so-called ``spaces of homogeneous type" (see Theorem \ref{thm-christ}). 
% which allows us to use  the   Calder\'on-Zygmund technique.   
%In this subsection quoted from \cite{Ca}, we first  introduce the Christ decomposition,  which allows us to use  the   Calder\'on-Zygmund technique  to deduce the power decay  for  the distribution function  of $|D^2u|.$
%$|\{u\geq h\}|$ of nonnegative supersolutions.    %Under the assumption that  a complete Riemannian manifold $M$ satisfies the volume doubling property,
%In \cite{Ch}, Christ generalized the  Euclidean dyadic decomposition for   so-called ``spaces of homogeneous type" (see Theorem \ref{thm-christ}).  
  In harmonic analysis, a metric measure space  $\cX=\left(\cX,d,\nu\right)$ is called  a space of homogeneous type when  a nonnegative Borel measure $\nu$ satisfies  the doubling property with a  doubling constant $\cD>0$: 
\begin{equation*}%\label{eq-VDP}
\nu(B_{2r}(x))\leq  \cD\, \nu(B_r(x)) <+\infty\quad\forall x\in \cX,\,\, r>0. 
\end{equation*}
%for some constant $\cD$ independent of $x$ and $R$.  
 A Riemannian manifold with nonnegative Ricci curvature has the volume doubling property with the  doubling constant  $\cD=2^n$.  When a Riemannian manifold $M$ has  a negative lower bound of the Ricci curvature,  the Riemannian measure of  $M$ has a locally uniform doubling  property; see  Bishop-Gromov's  Theorem \ref{thm-BG}.  As a matter of fact, one can see  that  the following Christ decomposition is   valid for the metric measure space equipped with   a local doubling measure. 
  
  \begin{comment}
%According to Bishop-Gromove's volume comparison theorem \ref{thm-BG}, a complete Riemannian manifold  $M$ satisfying the condition \eqref{cond-M-1} is a space of homogeneous type with $A_1=2^n$.
We say a metric measure space $\cX=(\cX,d,\nu)$ has a locally doubling  measure  if any
The measure ? is said to be locally doubling if for any fixed closed
ball B[z,R] ? X, there is a constant D = D(z,R) such that
?x ? B[z, R],	?r ? (0, R),	?[B2r](x)] ² D ?[Br](x)].

The proof of (ii)Ð(v) is as in [Ch]. In fact, the proof depends only on the metric
structure of the space and not on the properties of the measure µ and is even easier in our
case, because ? is a genuine distance, rather than a quasi-distance.
The proof of (i) is again as in [Ch]; observe that only a local doubling property is used in
the proof.
\end{comment}

 \begin{thm}[Christ]\label{thm-christ} Assume that 
the Ricci curvature of  $M$ is bounded  from below. % by a real number on $M.$
% $\Ric\geq-\kappa$ on $M$ for $\kappa\geq0$.  
 % Assume that  a metric measure space $\cX$ equips with   a locally uniform  doubling measure, i.e.,
There exist a countable collection $\left\{Q^{k,\alpha}\subset M : k\in\Z,\alpha\in I_k\right\}$ of open subsets of $M$ and positive uniform  constants $\delta_0\in(0,1)$, $c_1$ and $c_2$ (with $2c_1\leq c_2$ ) such that %which depend only on $\cD,$ such that
\begin{enumerate}[(i)]
\item $\left| M \backslash \bigcup_{\alpha\in I_k}Q^{k,\alpha}\right|=0$ for $k\in\Z$,
\item if $l\leq k$, $\alpha\in I_k$, and $\beta\in I_l$, then either $Q^{k,\alpha}\subset Q^{l,\beta}$ or $Q^{k,\alpha}\cap Q^{l,\beta}=\emptyset$,
\item for any $\al\in I_k$ and any $l<k$, there is a unique $\beta\in I_l$ such that $Q^{k,\alpha}\subset Q^{l,\beta}$,
\item $\diam(Q^{k,\alpha})\leq c_2\delta^k_0 $,
\item any $Q^{k,\alpha}$ contains some ball $B_{c_1\delta^k_0}(z^{k,\alpha})$.
\end{enumerate}
\end{thm}
%For convenience, we will use the following notation. 
%\begin{definition}[Dyadic cubes on  $M$]
%\label{def-d-cube-m}\item
 %\begin{enumerate}[(a)]
%\item  

The open set $Q^{k,\alpha}$ in Theorem \ref{thm-christ}  is called  a dyadic cube of generation $k$ on $M$. % due to the analogy between them and the standard Euclidean dyadic cubes. 
 The property (iii) asserts that     for any $\alpha\in I_k$,  there is  a unique $\beta\in I_{k-1}$ such that 
$Q^{k,\alpha}\subset Q^{k-1, \beta}$. We call $Q^{k-1,\beta}$ the predecessor of $Q^{k,\alpha}$ which is denoted by $\widetilde{Q^{k,\alpha}}$ for simplicity.

 For the rest of the paper, we fix  some small numbers;%\marginpar{\color{red}Choose small $\delta, \kappa$.}
\begin{equation}\label{eq-choice-delta}
\delta:=\frac{2c_1}{c_2}\delta_0\in(0,\delta_0),\quad\mbox{and}\quad\delta_1:= \frac{\delta_0(1-\delta_0)}{2}\in\left(0,\frac{\delta_0}{2}\right),%\quad \mbox{and}
%\quad \mbox{and}\quad \delta:=\delta\delta_0\in(0,1)
\end{equation}
%$${\delta}:=\min(\delta,\delta_1)\in\left(0,\frac{1}{2}\right), $$
where $\delta_0\in(0,1), c_1$ and $c_2$  are the constants  in Theorem \ref{thm-christ}.
For a given $R>0$,  we define $k_R\in\N$ to satisfy
\begin{equation*}
c_2\delta_0^{k_R-1}<R\leq c_2\delta_0^{k_R-2}.
\end{equation*}
% \end{enumerate}
%\end{definition}
The number $k_R$ means that a dyadic cube of generation $k_R$ is comparable to a ball of radius $R$.  
%Now  we fix  a  universal constant $0<\delta<1$  to satisfy the following: if $Q$ is any dyadic cube of generation $k,$ then there exist $\widetilde{z_k}\in Q$ and $\overline r_k>0$ such that 
% $$B_{\delta r_k}(\widetilde{z_k})\subset Q \subset\widetilde Q\subset B_{2r_k}(z_k). $$ 
 The following technical lemma  is quoted from    \cite[Lemma 6.5]{Ca}, which deals with  the relation between dyadic cubes and  comparable balls. 
\begin{lemma}\label{lem-decomp-cube-ball}
 Let $x_0\in M$ and $R>0$. %Then we have the following.
  \begin{enumerate}[(i)]
 \item  If $Q$ is a dyadic cube of generation $k$ such that
 $$k\geq k_R\quad\mbox{and}\quad Q\subset B_{2R}(x_0), $$
 then there exist $\overline z\in Q$ and $\overline r_k\in(0,R)$ such that
 \begin{equation}\label{eq-decomp-m-1}
 B_{\delta \overline r_k}(\overline z)\subset Q\subset\widetilde Q\subset \overline B_{2\overline  r_k}(\overline z)\subset B_{7\overline r_k}(\overline z)\subset B_{7R}(x_0) 
 \end{equation}
 and 
 \begin{equation}\label{eq-decomp-m-2}
 B_{5R}(x_0)\subset B_{7R}(\overline z).
 \end{equation} 
In fact, for $k\geq k_R,$  the   radius $\overline r_k$ is defined  by 
$$\overline r_k:=\frac{c_1}{\delta}\delta_0^k=\frac{1}{2}c_2\delta_0^{k-1}.$$ 
 \item If $Q$ is a dyadic cube of generation $k_R$ and $d(x_0, Q)\leq \delta_1R$,  then $Q\subset B_{2R}(x_0) $ and hence \eqref{eq-decomp-m-1} and \eqref{eq-decomp-m-2} hold for some $\overline z\in Q$ and $\overline r_k=\overline r_{k_R}=\frac{1}{2}c_2\delta_0^{k_R-1}\in \left[\frac{\delta_0 R}{2},\frac{R}{2}\right)$. Moreover, 
 $$B_{\delta_1R}(x_0)\subset B_{2\overline r_{k_R}}(\overline z). $$
 \item There exists at least one dyadic cube $Q$ of generation $k_R$ such that $d(x_0, Q)\leq\delta_1R$.  \end{enumerate}
\end{lemma}

%Before we prove the $W^{2,\delta}$ estimate, 
 Using  the Calder\'on-Zygmund technique, Lemma \ref{lem-CZ} follows from Theorem \ref{thm-christ};  the proof  can be found in   \cite[Lemma 6.3]{Ca}.  
\begin{lemma}\label{lem-CZ}
Let $Q_1\subset M$ be a dyadic cube, $A\subset B\subset Q_1$ be measurable sets, and $\sigma\in(0,1)$ satisfying  
\begin{enumerate}[(i)]
\item $|A|\leq \sigma|Q_1|$ and
\item if $Q\subset Q_1$ is a dyadic cube satisfying $|A\cap Q|>\sigma |Q|,$ then $\widetilde Q\subset B.$
\end{enumerate}
Then we have  $|A|\leq \sigma |B|.$
\end{lemma}

%%%%%%%%%%%%%%%%%%%%%%%%%%%%%%%%%%%%%%%%%%%%%%%%%%%%%%%%%%%%%%%%%%%%%%%%%%%%%%%%%%%%%%%%%%%%%%%%%%%%%%% section
%%%%%%%%%%%%%%%%%%%%%%%%%%%%%%%%%%%%%%%%%%%%%%%%%%%%%%%%%%%%%%%%%%%%%%%%%%%%%%%%%%%%%%%%%%%%%%%%%%%%%%
\subsection{Proof of $W^{2,\varepsilon}$-estimate}%Theorem \ref{thm-W2epsilon}}
%\subsection{Decay estimate}
%%%%%%%%%%%%%%%%%%%%%%%%%%%%%%%%%%%%%%%%%%%%%%%%%%%%%%%%%%%%%%%%%%%%%%%%%%%%%%%%%%%%%%%%%%%%%%%%%%%%%%% section
%%%%%%%%%%%%%%%%%%%%%%%%%%%%%%%%%%%%%%%%%%%%%%%%%%%%%%%%%%%%%%%%%%%%%%%%%%%%%%%%%%%%%%%%%%%%%%%%%%%%%%

This subsection is devoted to  the proof of a (locally) uniform  $W^{2,\ve}$-estimate for    uniformly elliptic  operators.   
 %  Applying  the ABP type estimate (Lemma \ref{lem-abp-type})   iteratively,   we 
 %by  obtaining a power decay of   the distribution function of $|D^2u|,$ the norm of   the Hessian of a solution $u$  in  Corollary \ref{cor-W2p}.     
Instead of the contact sets in Definition \ref{def-contact-set}, we introduce  the special  contact set  so as  to proceed with the ABP method   %applying  the ABP type estimate in 
using Lemma \ref{lem-abp-type} % iteratively 
in Proposition \ref{prop-decay-est-supersol}  with the help of   the Calder\'on-Zygmund technique.  
For the special contact set, we  make use of     sums of the barrier functions and the squared distance functions  as global  test       functions on a Riemannian manifold
since the  scale invariant barrier functions   are well-understood in Lemma \ref{lem-barrier}. With the choice of $\delta\in(0,1)$ in \eqref{eq-choice-delta},  the barrier function $v_\delta(\cdot;r;z)$ in Lemma \ref{lem-barrier}   will be denoted by $v(\cdot;r;z)$ %remvoing the subscript $\delta$ 
below and hereafter.
%here and in what follows. 

\begin{definition}[Special contact set]\label{def-special-contact -set}
{\rm 
Let $\Omega$ be a bounded, open set in $M.$ % For a given $a>0$ and a compact set $E\subset M,$ the (lower) contact set associated with $u$ of opening $a$ with the  vertex set $E$ is defined by
%Assume that  $\Sec\geq -\kappa$   for $\kappa\geq0.$ 
%For   $x_0\in M,$  $R>0, $ and $B_R(x_0)\subset M, $  let $u\in C\left(B_R(x_0)\right).$ % assume that $B_{7r}(z_0)\subset B_{R}(x_0).$ 
  %we use Lemma \ref{lem-CZ}  
  For $k\in \N,$
let  $\cQ_{k}\left(\Omega\right)$  be the  set  of global  test functions on $\Omega$ %which consist of  sums    the squared distance functions and the barrier functions centered in $\Omega$  
 %  let  $\cT_{K}\left(\Omega\right)$ be a class of   functions which consist of  sums    the squared distance functions and the barrier functions centered in $\Omega$  %up to the number of $l,$
%  where  the barrier functions are as
 %(in Lemma \ref{lem-barrier}), defined 
defined as   $$\cQ_{k}\left(\Omega\right):=\left\{ \sum_{j=1}^{k}r_j^2v(\cdot;r_j;z_j) +\sum_{j=1}^k\frac{1}{2} d_{y_j}^2 : \, r_j>0,\,\mbox{$B_{7r_j}(z_j)\subset \Omega, \, y_j\in \overline B_{7r_{j}}(z_j)\,\,\forall j=1,\cdots,k$   }\right\}.$$
  %Let $K\in\N.$ 
%\begin{enumerate}[(a)] \item 
  For  $u\in C(\Omega),$
and  $k\in \N,$   define the  {\it special  contact set } associated with $u$  of degree $k$ over $\Omega$
by  
\begin{align*}
&{\color{black}\underline\sG^{k}\left(u; \Omega\right)}:=\left\{x\in \Omega: 
\inf_{\Omega} \left( u+q\right)= u(x)+q(x)\,\,\,\mbox{for some  $q\in \cQ_{l}\left(\Omega\right)$ with $1\leq l\leq k.$
%$l\in(0,k]\cap\N $
} \right\}. 
\end{align*}
%\begin{align*}
%&{\color{black}\underline\sG^{K}\left(u; B_{R}(x_0)\right)}:=\\
%&\left\{x\in B_{R}(x_0): 
%\inf_{B_{R}(x_0)} \left( u+P\right)= u(x)+P(x)\quad\mbox{for   $P\in \cT_{l}\left(B_{R}(x_0)\right)$ with $l\in(0,K]\cap\N $} \right\},
%\end{align*}
%\mbox{$u$ is touched by a function in $\cT_{K}(B_R(z_0))$ from below at $x$ in $B_{7R}(z_0)$ up to translation} \right\}$$ 
%$${\color{blue}\underline\sG^{K}(u)}:=\left\{x\in B_{7R}(z_0): \mbox{$u$ is touched by a function in $\cT_{K}(B_R(z_0))$ from below at $x$ in $B_{7R}(z_0)$ up to translation} \right\}$$ 
   
% \item \begin{align*}
%&{\color{black}\overline \sG^{K}\left(u; B_{R}(x_0)\right)}:=\\&\left\{x\in B_{R}(x_0): \sup_{B_{R}(x_0)} \left( u-P\right)= u(x)-P(x)\quad\mbox{for   $P\in \cT_{l}\left(B_{R}(x_0)\right)$ with $l\in(0,K]\cap\N $} \right\},
%\end{align*}
%\item 
We also define 
$$
 \sG^{k}\left(u;  \Omega\right):=\underline\sG^{k}\left(u; \Omega\right)\cap \underline\sG^{k}\left(-u; \Omega\right).
$$
 %\end{enumerate}
% Here, the constant $\delta_0>0$ as in Theorem \ref{thm-christ} depends only on $\sqrt{\kappa}R_0.$
 }
 \end{definition}

%Making use of the notion of the special contact set, 

Now,  we obtain the following  power decay  estimate of the measure of the special contact set with respect to the degree, using   Lemma \ref{lem-abp-type}. % in some power.

\begin{prop}\label{prop-decay-est-supersol}
%Assume that  $\Sec\geq -\kappa$   for $\kappa\geq0.$ For  $z_0,x_0\in M$ and $0<r\leq R\leq R_0,$ assume that $  B_{7r}(z_0)\subset B_{R}(x_0).$   Let $u$ be a smooth function  on $ \overline  B_{R}(x_0)$ such that  $\cM^-(D^2u)\leq f$ in $B_{7r}(z_0),$
Assume that  $\Sec\geq -\kappa$   for $\kappa\geq0,$ and let    $x_0\in M$ and $0<7R\leq R_0.$ % assume that $B_{7R}(z_0)\subset B_{R_0}(x_0).$ 
%Let $u$ be a smooth function on $\overline B_{7R}(z_0)\subset B_{R_0}(x_0)$ such that    $\cM^{-}(D^2u)\leq f$
 %on $B_{7R}(z_0),$  $\|u\|_{L^{\infty}(\overline B_{7R}(z_0))}\leq1$ with $\left(\fint_{B_{7R}(z_0)} |R^2 f^+|^{n\eta}\right)^{\frac{1}{n{\color{black}\eta}}}\leq \e_\delta.$ 
  %  Let $M$ be a Riemannian  manifold with a lower bound $-\kappa$ ($\kappa\geq0$) of the sectional curvature.  
  Let $u$ be a smooth function in $\overline B_{7R}(x_0)$ such that  $$\|u\|_{L^{\infty}\left(B_{7R}(x_0)\right)}\leq1/2,$$ and     $u\in\overline \cS\left(\lambda,\Lambda, f\right)$ %$\cM^{-}(D^2u)\leq f$
 on $B_{7R}(x_0)$ with   
$$\left(\fint_{B_{7R}(x_0)} |R^2 f^+|^{n\eta}\right)^{\frac{1}{n\eta}}\leq \frac{\e}{2};\quad\ \eta:=1+\log_2\cosh (4\sqrt\kappa R_0).$$ 
Let $Q_1$ be a dyadic cube of generation $k_R$ such that $d(x_0,Q_1)\leq \delta_1 R,$ and let $\overline r_{k_R}\in \left[\frac{\delta_0 R}{2},\frac{R}{2}\right)$ be the radius in Lemma \ref{lem-decomp-cube-ball} (ii).      
Then we have 
\begin{equation}\label{eq-decay-est}
\frac{\left|Q_1\setminus \underline\sG^{K^{(i-1)}}\left(\overline r_{k_R}^2u; B_{7R}(x_0)\right) \right|}{|Q_1|}\leq C\left(1-\frac{\mu}{2}\right)^i\quad\forall i=1,2,\cdots,
\end{equation}  
%for $\sigma\in(0,1),$
where     the uniform constants $K\in \N$ and $C\geq1$ depend only on $n,\lambda,\Lambda,$ and $\sqrt{\kappa}R_0,$ and 
 the   constants  $\e:=\e_\delta, \mu:=\mu_\delta$  are as in Lemma \ref{lem-abp-type} with \eqref{eq-choice-delta}. %, and   %As a consequence, 
%\begin{equation*}%\label{eq-decay-est}
%\frac{\left|Q_1\setminus \underline\cG_{h}(R^2u) \right|}{|Q_1|}\leq dh^{-\ve}\quad\forall h>0. 
%\end{equation*}
\end{prop}
\begin{proof}
(i) 
%We use the induction for $i\in \N,$ and a   large number $K\in\N$ will be chosen  later. First, let $i=1.$  
First, we prove 
\begin{equation}\label{eq-decay-est-1st}
\frac{\left|Q_1\setminus \underline\sG^{1}\left(\overline r_{k_R}^2u; B_{7R}(x_0)\right) \right|}{|Q_1|}\leq 1-\mu.
\end{equation}  
In fact,  from Lemma \ref{lem-decomp-cube-ball}, we find  $\overline z_{k_R}\in Q_1$ and  $\overline r_{k_R}\in \left[\frac{\delta_0 R}{2},\frac{R}{2}\right)$ for $k=k_R$ such that 
\begin{equation}\label{eq-Q_1-ball}
B_{\delta \overline r_{k_R} }(\overline z_{k_R})\subset  Q_1 \subset \widetilde Q_1\subset \overline B_{2\overline r_{k_R}}(\overline z_{k_R})\subset B_{7\overline r_{k_R}}(\overline z_{k_R})\subset B_{7R}(x_0).
\end{equation}
  We notice that   $0\leq u+\frac{1}{2}\leq 1$ in $B_{7R}(x_0)$ and recall from Lemma \ref{lem-weight-int-ellip} that
   $$\left\{\fint_{B_{7\overline r_{k_R}}(\overline z_{k_R})}\left|\overline r_{k_R}^2f^+\right|^{n\eta}\right\}^{\frac{1}{n\eta}}\leq  2 \left\{\fint_{B_{7R}(x_0)}\left|R^2f^+\right|^{n\eta}\right\}^{\frac{1}{n\eta}}\leq \e.%;\qquad \eta:= \frac{1}{n}\log_2\cD.
$$
   Thus  we apply   Lemma \ref{lem-abp-type} to $u+\frac{1}{2}$ in order to obtain
\begin{align*}%\label{eq-abp-type}
  \frac{\left| \underline\cG_{\overline r_{k_R}^{-2}}\left(u+\frac{1}{2}+v(\cdot;\overline r_{k_R};\overline z_{k_R}); \overline B_{7\overline r_{k_R}}(\overline z_{k_R}); B_{7R}(x_0)\right)\cap B_{\delta \overline r_{k_R}}(\overline z_{k_R}) \right|}{\left|B_{7\overline r_{k_R}}(\overline z_{k_R})\right| } \geq\mu.
 \end{align*}  
 %$$\mu< \frac{\left|\underline\cG_{M}(R^2u)\cap B_{\delta R}(z_k)\right|}{|B_{7r_k}(z_k)|} \leq \frac{\left|\underline\cG_{M}(R^2u)\cap Q_1 \right|}{|Q_1|},$$ 
%The definition of $\underline\sG^{1}\left(\overline r_{k_R}^2u; B_{7R}(z_0)\right)$ yields that 
Since we have $$ \underline\cG_{\overline r_{k_R}^{-2}}\left(u+\frac{1}{2}+v(\cdot;\overline r_{k_R};\overline z_{k_R}); \overline B_{7\overline r_{k_R}}(\overline z_{k_R}); B_{7R}(x_0)\right) \subset \underline\sG^{1}\left(\overline r_{k_R}^2u; B_{7R}(x_0)\right),$$
we deduce that 
%Thus \eqref{eq-decay-est} holds for $i=1$  since 
\begin{equation*}%\label{eq-decay-est-1st}
 \frac{\left| \underline\sG^{1}\left(\overline r_{k_R}^2u; B_{7R}(x_0)\right)\cap Q_1\right|}{\left|Q_1\right| }  \geq  \frac{\left|\underline\sG^{1}\left(\overline r_{k_R}^2u; B_{7R}(x_0)\right)\cap B_{\delta\overline  r_{k_R}}(\overline z_{k_R}) \right|}{\left|B_{7\overline r_{k_R}}(\overline z_{k_R})\right| }
\geq\mu,
\end{equation*}
which implies \eqref{eq-decay-est-1st}. % for $i=1.$

(ii) %Now we assume that \eqref{eq-decay-est} holds for  some $i\in N,$ and let 
For $i\in\N,$ we define 
$$A:= Q_1\setminus \underline\sG^{K^{i}}\left(\overline r_{k_R}^2u; B_{7R}(x_0)\right),$$
and $$ B:= \left(Q_1\setminus \underline\sG^{K^{i-1}}\left(\overline r_{k_R}^2u; B_{7R}(x_0)\right)\right)\cup \left\{z\in Q_1: m_{B_{7R}(x_0)}\left( |R^2f^+|^{n\eta}\right)(z)> K^{(i-1)n\eta} \right\},$$
where $m_{B_{7R}(x_0)}$   is the  maximal operator  over $B_{7R}(x_0)$ in  Lemma \ref{lem-max-ft}.
% is defined  in  $$m_{B_{r}(x)} \left(g\right)(z):=\sup_{z\in B_{s}(y)\subset B_{r}(x)}\fint_{B_s(y)}|g|\, d\vol\quad\mbox{for  an integrable $g$ in $B_r(x).$}$$  
 It is clear  that $A\subset B\subset Q_1$ and $|A|\leq (1-\mu)|Q_1|$ according to \eqref{eq-decay-est-1st}.   Now we claim that 
\begin{equation}\label{eq-decay-iteration}
|A|\leq (1-\mu)|B|,
\end{equation}
for  a   large number $K\in\N$ to be  chosen  later. 
 %which will complete the proof.  
Using  Lemma \ref{lem-CZ},   it suffices to show that if $Q\subset Q_1$ is a dyadic cube of generation $k> k_R$ such that   $|A\cap Q|> (1-\mu)|Q|,$ then  $\widetilde Q\subset B. $ 
Suppose to the contrary that $\widetilde Q\not\subset B,$ 
and  let   $ \tilde  z$ be a point  in $\widetilde Q\setminus B\subset Q_1.$
%Note that $Q\not=Q_1$ and $k>k_R$ since $|A\cap Q_1|\leq \left|Q_1\setminus \underline\sG^{1}\left(\overline r_{k_R}^2u\right) \right|\leq(1-\mu)|Q_1|$ from (i).   
%Since $\widetilde Q\not\subset B$ and $\widetilde Q\subset  Q_1,$ 
From Lemma \ref{lem-decomp-cube-ball} again,   we find  $  z_0\in  Q$ and $ \overline  r_k\in(0,R)$ such that   
\begin{align*}
  \tilde z &\in \widetilde Q \cap \underline\sG^{K^{i-1}}\left(\overline r_{k_R}^2u; B_{7R}(x_0)\right)\cap  \left\{z\in Q_1: m_{B_{7R}(x_0)}\left(|R^2f^+|^{n\eta}\right)(z)\leq K^{(i-1)n\eta} \right\} \\%\subset Q_1\cap \underline\cG_{M^{i}}(R^2u) 
&\subset \overline B_{2\overline r_k}(z_0)\cap \underline\sG^{K^{i-1}}\left(\overline r_{k_R}^2u; B_{7R}(x_0)\right)\cap  \left\{z\in Q_1: m_{B_{7R}(x_0)}\left( |R^2f^+|^{n\eta}\right)(z)\leq  K^{(i-1)n\eta} \right\},
\end{align*}
and   
$B_{\delta \overline r_k}(z_0)\subset Q\subset\widetilde Q\subset \overline B_{2\overline r_k}(z_0)\subset B_{7\overline r_k}(z_0)\subset B_{7R}(x_0) .$ 
From the definition of $\underline\sG^{K^{i-1}}\left(\overline r_{k_R}^2u; B_{7R}(x_0)\right),$ 
we consider  
$$\tilde u:= \overline r_{k_R}^2u+ q\qquad\mbox{in $B_{7R}(x_0)$}$$
for some $q\in\cQ_l\left(B_{7R}(x_0)\right)$ with $1\leq l\leq K^{i-1},$    
  satisfying  
 $$\inf_{B_{7R}(x_0)}\tilde u=\tilde u\left( \tilde z\right).$$
Indeed, we can write 
$$\tilde u:= \overline r_{k_R}^2u+\sum_{j=1}^{l}r_j^2v(\cdot;r_j;z_j) +\sum_{j=1}^{l}\frac{1}{2}d_{y_j}^2\quad\mbox{in $B_{7R}(x_0),$}$$
%there exist $r_j, z_j, $ and $y_j$ for $j=1,2,\cdots,l$
  where $r_j>0,$ $B_{7r_j}(z_j)\subset B_{7R}(x_0)$ and $y_j\in\overline B_{7r_j}(z_j)$ for  $j=1,\cdots,l.$
 \begin{comment}
\begin{equation*}
\begin{split}
&B_{7r_j}(z_j)\subset B_{7R}(z_0)\quad y_j\in\overline B_{7r_j}(z_j)\\
&\tilde u:= \overline r_{k_R}^2u+\sum_{j=1}^{l}r_j^2v(\cdot;r_j;z_j) +\sum_{j=1}^{l}\frac{1}{2}d_{y_j}^2+\tilde C\geq 0\quad\mbox{in $B_{7R}(z_0)$}\quad\mbox{and }\quad\\
& \tilde u(z^*)=0 \quad\mbox{for some $\tilde C\in\R$}. 
\end{split}
\end{equation*}
\end{comment}
 For  large constants $K>K_0>1,$ define  
 $$\tilde w:= \frac{1}{  \overline r_k^2K_0K^{i-1}} \tilde u - \frac{1}{  \overline r_k^2K_0K^{i-1}}\tilde u\left( \tilde z\right),$$ which is nonnegative in   $B_{7R}(x_0) ,$ and vanishes at $ \tilde  z \in \overline B_{2  \overline r_k}(z_0).$ 
 Recall from Lemma \ref{lem-barrier} that 
$$r_j^2D^2v(\cdot;r_j;z_j)\leq C{\bf I}\quad\mbox{in $B_{7R}(x_0)\setminus \Cut(z_j)$}\quad\forall j=1,2,\cdots,l$$ 
for a uniform constant $C>0,$  depending only on $n,\lambda,\Lambda,$ and $\sH\left(2\sqrt{\kappa}R_0\right),$ 
%which is independent of $r_j$ and $z_j$   
 since  $B_{7r_j}(z_j)\subset B_{7R}(x_0)$ with $7R\leq R_0.$ 
This combined with  Lemma \ref{lem-hess-dist-sqrd}  implies 
 that 
 \begin{align*}
 \cM^-(D^2 \tilde w)&\leq \frac{1}{  \overline r_k^2K_0K^{i-1}}\overline r_{k_R}^2f^++\frac{1}{  \overline r_k^2K_0K^{i-1}}  \left\{\sum_{j=1}^{l}r_j^2\cM^+\left(D^2v(\cdot;r_j;z_j)\right) + \sum_{j=1}^{l}\cM^+\left(D^2d_{y_j}^2/2\right)\right\}\\
 &\leq \frac{1}{  \overline r_k^2K_0K^{i-1}}\overline r_{k_R}^2f^++\frac{1}{  \overline r_k^2K_0K^{i-1}}  K^{i-1}n\Lambda\left\{C+\sH\left(2\sqrt{\kappa}R_0\right) \right\}\\
 &=:\tilde f \qquad\qquad\qquad\mbox{in $B_{7R}(x_0)\setminus\left(\bigcup_{j=1}^l\Cut(z_j)\cup\bigcup_{j=1}^l\Cut(y_j)\right).$}
 \end{align*}
% for a uniform constant $C>0.$ % depending only on $n,\lambda,\Lambda,$ and $\sqrt{\kappa}R_0.$
 %We note that $r_j^2D^2v(\cdot;r_j;z_j)\leq C_\delta{\bf I}$  in $B_{7R}(z_0)$ for a uniform constant $C_\delta\geq 1$ which is independent of $r_j$ and $z_j$ such that $B_{7r_j}(z_j)\subset B_{7R}(z_0).$ 
By  choosing $K_0\in\N $ sufficiently large, we deduce that 
% $\frac{C }{K_0}\leq \frac{\epsilon}{2}$ 
 %and hence we have 
 \begin{align*}
 \left(\fint_{B_{7  \overline r_k}(z_0)} |   \overline r_k^2 \tilde f|^{n\eta}\right)^{\frac{1}{n\eta}}
 &\leq \frac{1}{K_0K^{i-1}}\left(\fint_{B_{7\overline r_k}(  z_0)} | \overline r_{k_R}^2   f^+|^{n\eta}\right)^{\frac{1}{n\eta}}+\frac{1}{K_0}n\Lambda\left\{C+\sH\left(2\sqrt{\kappa}R_0\right) \right\}\\
 & \leq  \frac{1}{K_0K^{i-1}}\left(\fint_{B_{7\overline r_k}(  z_0)} | R^2   f^+|^{n\eta}\right)^{\frac{1}{n\eta}}+\frac{\e}{2}\\
 &\leq \frac{1}{K_0K^{i-1}}K^{i-1}+\frac{\e}{2}= \frac{1}{K_0}+\frac{\e}{2}
 \leq  \e
 \end{align*} %\quad\mbox{and}\quad \underline\cG_{\overline r_k^{-2}}(\tilde w)\cap B_{2\overline r_k}(z_k)\not=\emptyset.$$ 
 since $\tilde z\in \overline B_{2\overline r_k}(z_0)\cap    \left\{z\in Q_1: m_{B_{7R}(x_0)}\left(|R^2f^+|^{n\eta}\right)(z)\leq K^{(i-1)n\eta} \right\}.$ 
 Thus, we apply   Corollary \ref{cor-abp-type} to $\tilde w$ in $B_{7  \overline r_k}(z_0)$ to  obtain  that 
 %$$\frac{\left|\underline\cG_{\tilde  r^{-2}}\left(\tilde w+v(\cdot; \tilde r;\tilde z)\right)\cap B_{\delta \tilde r}(\tilde z)\right|}{ |B_{7R}(z_0)|} \geq \mu>0,$$
   \begin{equation}\label{eq-decay-est-scaled}
   \frac{\left|\underline\cG_{{  \overline r_k}^{-2}}\left(\tilde w+v(;  \overline r_k;   z_0); \overline B_{7  \overline r_k}(  z_0); B_{7R}(x_0)\right)\cap B_{\delta   \overline r_k}(  z_0)\right|}{ |B_{7  \overline r_k}(z_0)|} \geq \mu.
   \end{equation}

   Now we claim that 
  \begin{equation}\label{eq-good-sets}
  \underline\cG_{{  \overline r_k}^{-2}}\left(\tilde w+v(;  \overline r_k;   z_0); \overline B_{7  \overline r_k}(  z_0); B_{7R}(x_0)\right)\subset \underline\sG^{K^{i}}\left(\overline r_{k_R}^2u; B_{7R}(x_0)\right) 
  \end{equation}
   for a large uniform constant $K\in\N. $  In fact, 
 let  $x\in \cG_{{  \overline r_k}^{-2}}\left(\tilde w+v(;  \overline r_k;   z_0); \overline B_{7  \overline r_k}(  z_0); B_{7R}(x_0)\right).$ Then  
 we have 
 \begin{equation*}
\begin{split}
\inf_{B_{7R}(x_0)}\left\{ \tilde w+v(\cdot;  \overline r_k;   z_0)+ \frac{1}{2  \overline r_k^2} d_{  y_0}^2\right\}= \tilde w(x)+v(x;  \overline r_k;   z_0)+  \frac{1}{2  \overline r_k^2}d_{  y_0}^2(x)
\end{split}
\end{equation*} for some $  y_0\in \overline B_{7  \overline r_k}(  z_0),$
that is, 
$$ \overline r_{k_R}^2u+ \sum_{j=1}^{l}r_j^2v(\cdot;r_j;z_j)+K_0K^{i-1}   \overline r_k^2v(\cdot;  \overline r_k;  z_0) +\sum_{j=1}^{l}\frac{1}{2} d_{y_j}^2+\frac{K_0K^{i-1} }{2} d_{  y_0}^2$$
has its minimum at $x$ in $B_{7R}(x_0).$ 
 By choosing $K\in\N$ such that $K\geq1+K_0,$ we conclude  that $x\in \underline\sG^{K^{i}}\left(\overline r_{k_R}^2u; B_{7R}(x_0)\right) ,$ which proves \eqref{eq-good-sets}.  

% $ \overline r_{k_R}^2u$   lies above and is  touched by 
%$$ - \sum_{j=1}^{l}r_j^2v(\cdot;r_j;z_j)-K_0K^{i-1} \tilde r^2v(\cdot;\tilde r;\tilde z) - \sum_{j=1}^{l}\frac{1}{2} d_{y_j}^2-K_0K^{i-1} \frac{1}{2} d_{\tilde y}^2$$
%from below  at $x$ up to translation. By choosing $K\in\N$ such that $K>1+K_0,$ we deduce that $x\in \underline\sG^{K^{i}}(\overline r_{k_R}^2u).$ 
     Together with \eqref{eq-decay-est-scaled} and \eqref{eq-good-sets},  we have  that 
       $$\frac{\left|\underline\sG^{K^i}\left(\overline r_{k_R}^2u;B_{7R}(x_0)\right)\cap Q\right|}{ |Q|}\geq \frac{\left|\underline\sG^{K^i}\left(\overline r_{k_R}^2u;B_{7R}(x_0)\right)\cap B_{\delta   \overline r_k}(  z_0)\right|}{ |B_{7  \overline r_k}(  z_0)|} \geq \mu,$$
  % $$\frac{\left|\underline\cG_{M^{i+1}R^{-2}}(u)\cap B_{\delta \overline r_k}(z_k)\right|}{|B_{7\overline r_k}(z_k)|}=\frac{\left|\underline\cG_{M\overline r_k^{-2}}(\tilde u)\cap B_{\delta \overline r_k}\right|}{|B_{7\overline r_k}(z_k)|}>\mu.$$
which means that   $ {|A\cap Q|}<(1-\mu)|Q|.$
This contradicts to the assumption  $|A\cap Q|>(1-\mu)|Q|.$ Therefore, we have proved   \eqref{eq-decay-iteration}  according to 
Lemma \ref{lem-CZ}. % yields that $|A|\leq(1-\mu)|B|,$ completing the proof of \eqref{eq-decay-iteration}. 

(iii)
Let 
$$\alpha_i=\frac{\left|Q_1\setminus \underline\sG^{K^{i}}\left(\overline r_{k_R}^2u; B_{7R}(x_0)\right)\right|}{|Q_1|},\quad\mbox{and}\quad \beta_i:=\frac{\left|\left\{z\in Q_1: m_{B_{7R}(x_0)}\left(|R^2f^+|^{n\eta}\right)(z)> K^{in\eta} \right\}\right|}{|Q_1|}.$$
 From (ii),  we have 
 $$\alpha_{i}\leq (1-\mu)(\alpha_{i-1}+\beta_{i-1})\quad\forall i\in\N$$ which implies that
 $$\alpha_i\leq (1-\mu)^i\alpha_0+ \sum_{j=0}^{i-1}(1-\mu)^{i-j}\beta_j\leq (1-\mu)^i+ \sum_{j=0}^{i-1}(1-\mu)^{i-j}\beta_j.$$
 Since $\int_{B_{7R}(x_0)} |R^2f^+|^{n\eta} \leq \e^{n\eta}|B_{7R}(x_0)|,$ the weak type $(1,1)$ estimate  in Lemma \ref{lem-max-ft}    leads to that 
 \begin{align*}
 \beta_j&=\frac{\left|\left\{z\in Q_1: m_{B_{7R}(x_0)}\left(|R^2f^+|^{n\eta}\right)(z)> K^{jn\eta} \right\}\right|}{|Q_1|} \leq  \frac{C_1}{K^{jn\eta}}\e^{n\eta}\frac{|B_{7R}(x_0)|}{|Q_1|}.
%&\leq C_1K^{-jn\eta}.
 \end{align*}
From  Lemma \ref{lem-decomp-cube-ball} and \eqref{eq-Q_1-ball},  we have  
$$\frac{|B_{7R}(x_0)|}{|Q_1|}=\frac{|B_{\delta_1 R}(x_0)|}{|Q_1|}\frac{|B_{7R}(x_0)|}{|B_{\delta_1R}(x_0)|}\leq \frac{|B_{2 \overline r_{k_R} }(\overline z_{k_R})|}{|B_{\delta \overline r_{k_R} }(\overline z_{k_R})|}\frac{|B_{7R}(x_0)|}{|B_{\delta_1R}(x_0)|},$$
%since $d(x_0,Q_1)\leq \delta_1 R$
 which   is %$|B_{7R}(x_0)|/|Q_1|$ is   
bounded by a uniform constant depending only on $n,$ and $\sqrt{\kappa} R_0$ from Bishop-Gromov's Theorem \ref{thm-BG}. Hence we have  that $\beta_j\leq C K^{-jn\eta}$ for  a uniform constant $C>0.$
Therefore we select $K> (1-\mu)^{-\frac{1}{n\eta}}$ sufficiently  large so that 
$$\alpha_i\leq \left(1+ Ci\right)(1-\mu)^i\quad\forall i\in\N,$$
which implies \eqref{eq-decay-est}.
\end{proof}
 
% \subsection{End of proof of $W^{2,\ve}$ estimate}

 Proposition \ref{prop-decay-est-supersol} yields a power decay estimate \eqref{eq-decay-est-1} of the distribution function of $|D^2u|,$ the norm of the Hessian of  the solution $u$ to the uniformly elliptic equation. 
 \begin{lemma}[Decay estimate]\label{lem-W2p}
%Let $M$ be a Hadamard manifold with a lower bound $-\kappa$ ($\kappa\geq0$) of the sectional curvature.  Let $u$ be a smooth function satisfying $Lu= f$ in $B_{7R}(z_0)$ such that$$\left(\fint_{B_{7R}} |R^2 f|^{n\eta}\right)^{1/n\eta}\leq \e_0/2\quad\mbox{and} \quad\|u\|_{L^{\infty}(B_{7R})}\leq1.$$ Let $Q_1$ be a dyadic cube of generation $k_R$ such that $d(z_0,Q_1)\leq \delta R.$   
Assume that  $\Sec\geq -\kappa$   for $\kappa\geq0.$ 
Let  $x_0\in M$ and $0<7R\leq R_0.$  % assume that $B_{7R}(z_0)\subset B_{R_0}(x_0).$ 
  Let $u$ be a smooth function in $\overline B_{7R}(x_0)$ such that $\|u\|_{L^{\infty}(B_{7R}(x_0))}\leq1/2,$ and   $u\in\cS^*\left(\lambda,\Lambda,f\right)$  % $\cM^{-}(D^2u)\leq f^+$ and $-f^-\leq \cM^+(D^2u)$
 in $B_{7R}(x_0)$    with     
$$\left(\fint_{B_{7R}(x_0)} |R^2 f|^{n\eta}\right)^{\frac{1}{n\eta}}\leq \e/2; \quad\eta:=1+\log_2\cosh (4\sqrt\kappa R_0).$$ 
Let $Q_1$ be a dyadic cube of generation $k_R$ such that $d(x_0,Q_1)\leq \delta_1 R.$   
%Then 
%Under the same assumption as Proposition \ref{prop-decay-est-supersol}, 
%we have 
%\begin{equation}\label{eq-decay-est-1}
%\frac{\left|\left\{-R^2D^2u> h{\bf I}\right\}\cap Q_1\right|}{|Q_1|}  \leq C h^{-\e_0}  \quad\forall h>0
%\end{equation} 
%for uniform constants $\e_0\in(0,1)$ and $C>0$ depending only on $n,\lambda,\Lambda,$ and $ \sqrt{\kappa}R_0.$ 
% In addition, 
%if  $$ in $B_{7R}(x_0),$  
Then we have 
\begin{equation}\label{eq-decay-est-1}
 {\left|\left\{R^2|D^2u|\geq h\right\}\cap Q_1\right|}  \leq C h^{-\e_0}|Q_1|  \quad\forall h>0
 \end{equation} 
 for uniform constants $\e_0\in(0,1)$ and $C>0$ depending only on $n,\lambda,\Lambda,$ and $ \sqrt{\kappa}R_0,$ 
and hence 
\begin{equation*}%\label{eq-hess-bound}
\left(\fint_{Q_1} \left|R^2 D^2u\right|^{\ve}\right)^{\frac{1}{\ve}}\leq C,
\end{equation*}
where $\ve\in(0,1)$ and $C>0$ are uniform constants.
\end{lemma}

\begin{proof}
Applying    Proposition \ref{prop-decay-est-supersol} to $\pm u$, we have 
\begin{equation}\label{eq-decay-est-sol}
\frac{\left|Q_1\setminus   \sG^{K^{(i-1)}}\left(\overline r_{k_R}^2u; B_{7R}(x_0)\right) \right|}{|Q_1|}\leq C\left(1-\frac{\mu}{2}\right)^i\quad\forall i=1,2,\cdots.
%\frac{\left|Q_1\setminus  \sG^{K^{(i-1)}}(\overline r_{k_R}^2u) \right|}{|Q_1|}\leq C \left(1-\frac{\mu}{2}\right)^i\quad\forall i=1,2,3,\cdots,
\end{equation}  
In order to prove \eqref{eq-decay-est-1}, we    claim that   for any $i\in\N\cup\{0\}$
\begin{equation}\label{eq-hess-parabola}
    \sG^{K^{i}}\left(\overline r_{k_R}^2u; B_{7R}(x_0)\right)\subset  \left\{x\in B_{7R}(x_0): R^2|D^2u(x)| \leq C K^{i}\right\},
\end{equation} 
where $C>0$ is a uniform constant depending only on $n,\lambda,\Lambda,$ and $\sqrt{\kappa}R_0.$
In fact, fix $i\in\N\cup\{0\},$ and   $x\in \sG^{K^i}\left(\overline r_{k_R}^2u; B_{7R}(x_0)\right).$ Then 
there exist  $ q_1\in \cQ_{l_{1}}\left(B_{7R}(x_0)\right),$ and $q_2\in \cQ_{l_{2}}\left(B_{7R}(x_0)\right)$ for $1\leq l_1,l_2\leq K^{i}$    
  such that %\eqref{eq-contact-pt-above-below}
  \begin{equation}\label{eq-contact-pt-above-below}
\begin{split}
-  q_1+  q_1(x)%&:=-\sum_{j=1}^{l_-}  r_j^2v(\cdot;     r_j;  z_j) -\sum_{j=1}^{l_-} \frac{1}{2} d_{     y_j}^2 +\sum_{j=1}^{l_-}     r_j^2v(x;      r_j;  z_j) +\sum_{j=1}^{l_-} \frac{1}{2}d_{     y_j}^2(x)\\ 
&\leq   \overline r_{k_R}^2u -\overline r_{k_R}^2u(x)% \geq -  q_1+  q_1(x)
%&\leq \sum_{j=1}^{l_+} \overline r_j^2v(\cdot; \overline r_j;\overline z_j) +\sum_{j=1}^{l_+}\frac{1}{2}d_{\overline y_j}^2 -%\sum_{j=1}^{l_+} \overline r_j^2v(x; \overline r_j;\overline z_j) -\sum_{j=1}^{l_+}\frac{1}{2} d_{\overline y_j}^2(x) \\
\leq   q_2-  q_2(x)
 \qquad\mbox{in $B_{7R}(x_0)$}.
 \end{split}
\end{equation}
From the definition of  of $\cQ_{l_{1}}\left(B_{7R}(x_0)\right), $ we can write 
\begin{equation}\label{eq-rep-q_1}
\begin{split}
-  q_1+  q_1(x)&:=-\sum_{j=1}^{l_1}  r_j^2v(\cdot;     r_j;  z_j) -\sum_{j=1}^{l_1} \frac{1}{2} d_{     y_j}^2 +\sum_{j=1}^{l_1}     r_j^2v(x;      r_j;  z_j) +\sum_{j=1}^{l_1} \frac{1}{2}d_{     y_j}^2(x)\\ 
&\leq  \overline r_{k_R}^2u -\overline r_{k_R}^2u(x)  \qquad\mbox{in $B_{7R}(x_0)$}
\end{split}
\end{equation}
for some  $r_j>0,$ $B_{7{  r}_j}(  z_j)\subset B_{7R}(x_0)$ and ${  y}_j\in\overline B_{7  r_j}(  z_j)\,\, (j=1,2,\cdots,l_1).$  For a unit vector  $X\in T_xM,$ % with $|X|=1,$
 we have 
\begin{align*}
%-\sH(7\sqrt{\kappa}R) h &\leq -h\sqrt{\kappa}d_{y_1}(x)\coth\left(\sqrt{\kappa}d_{y_1}(x)\right)\\
%&\leq - \limsup_{r\to0}\frac{h}{2}\frac{d_{y_1}^2\left(\exp_xrX\right)+d_{y_1}^2\left(\exp_x -rX\right)-2d^2_{y_1}(x)}{r^2}\\
&  \liminf_{t\to 0}\frac{\overline r_{k_R}^2u\left(\exp_xtX\right)+\overline r_{k_R}^2u\left(\exp_x -tX\right)-2\overline r_{k_R}^2u(x)}{t^2}\\
\geq&-\sum_{j=1}^{l_1}  r_j^2  \limsup_{t\to0} \frac{1}{t^2}\left\{ v(\exp_xtX;  r_j;  z_j)+v(\exp_x-tX;  r_j;  z_j)-2v(x;  r_j;  z_j)  \right\}\\
&- \sum_{j=1}^{l_1}\limsup_{t\to0} \frac{1}{2t^2}\left\{d^2_{  y_j}(\exp_xtX)+d_{  y_j}^2(\exp_x-tX)-2d^2_{  y_j}(x)\right\}.
%&\leq  2C K^i
\end{align*} 
We recall from Remark \ref{rmk-barrier} and Lemma \ref{lem-hess-dist-sqrd} to obtain that 
$$\liminf_{t\to 0}\frac{\overline r_{k_R}^2u\left(\exp_xtX\right)+\overline r_{k_R}^2u\left(\exp_x -tX\right)-2\overline r_{k_R}^2u(x)}{t^2}\geq -Cl_1\geq -CK^{i}$$ for a uniform constant $C>0. $
\begin{comment}
that  for all $j=1,\cdots,l_-$
\begin{align*}
\underline r_j^2D^2v(\cdot;\underline r_j;\underline z_j) \leq C {\bf I} \quad\mbox{in $B_{7R}(x_0)\setminus \bigcup_{j=1}^{l_-}\Cut(\underline z_j)$}\\
D^2d^2_{\underline y_j} \leq C {\bf I} \quad\mbox{in $B_{7R}(x_0)\setminus \bigcup_{j=1}^{l_-}\Cut(\underline y_j)$}
\end{align*}
for  a uniform constant $C>0$ since  $\underline y_j\in B_{7\underline r_j}(z_j)\subset B_{7R}(x_0),$ 
where we note that $$x\not\in \bigcup_{j=0}^{l_-}\Cut(\underline z_j)\cup \bigcup_{j=0}^{l_-}\Cut(\underline y_j) $$ from Lemma \ref{lem-contact-set-tech}.  % where $ \bigcup_{j=1}^{K^i}\Cut(z_j)\cup  \bigcup_{j=1}^{K^i}\Cut(y_j)$ has measure zero. 
\end{comment} 
 % According to our argument above using \eqref{eq-contact-pt-above-below}, 
 Thus  we deduce  that for any unit vector $X\in T_xM,$ % with $|X|=1,$
 $$ \overline r_{k_R}^2\left\langle D^2u(x)\cdot X,X\right\rangle\geq -CK^{i}.$$
 % $$-CK^{i}\leq \overline r_{k_R}^2\left\langle D^2u(x)\cdot X,X\right\rangle
   %\lim_{t\to 0}\frac{\overline r_{k_R}^2u\left(\exp_xtX\right)+\overline r_{k_R}^2u\left(\exp_x -tX\right)-2\overline r_{k_R}^2u(x)}{t^2}  \leq CK^{i}$$ 
   %for a uniform constant $C>0.$ 
   According to our argument above using \eqref{eq-contact-pt-above-below}, 
  we obtain    \eqref{eq-hess-parabola}  
since $\overline r_{k_R}\in \left(\frac{\delta_0R}{2},\frac{R}{2}\right].$     
Therefore, it follows from \eqref{eq-decay-est-sol}  that   for $i\in\N\cup\{0\}, $
$$\left|\left\{R^2|D^2u|> CK^i\right\}\cap Q_1\right|\leq \left|Q_1\setminus \sG^{K^i}\left(\overline r_{k_R}^2u; B_{7R}(x_0)\right)\right|\leq C\left(1-\frac{\mu}{2}\right)^{i+1}|Q_1|,$$
which implies  \eqref{eq-decay-est-1}.

Letting $\ve:=\e_0/2,$ we have 
\begin{align*}
\fint_{Q_1}\left(R^2|D^2u|\right)^\ve \, d\vol&=\frac{1}{| Q_1 |}\left\{\int_{Q_1\cap \{R^2|D^2u|\leq1\}}+\int_ {Q_1\cap \{R^2|D^2u|>1\}}\right\}\left(R^2|D^2u|\right)^\ve \, d\vol\\
&\leq 1+\frac{1}{|Q_1|}\ve\int_{1}^{\infty}h^{\ve-1}\left|\left\{|R^2D^2u|\geq h\right\}\cap Q_1\right|dh\\
&= 1+\frac{\e_0}{2}\int_{1}^{\infty}h^{\e_0/2-1}C h^{-\e_0}dh,
\end{align*}
where the last quantity is bound by a uniform constant. This completes the proof. 
  %proves \eqref{eq-hess-bound}
\end{proof}

 \begin{cor}%\label{cor-W1p}
Under the same assumption as Lemma \ref{lem-W2p}, % let $u\in\cS^*(f)$ in $B_{7R}(x_0).$ Then  
we have 
\begin{equation*} 
\left(\fint_{Q_1} \left|R \D u\right|^{\ve}\right)^{\frac{1}{\ve}}\leq C,
\end{equation*}
where $\ve\in(0,1)$ and $C>0$ are uniform constants.
\end{cor}
\begin{proof}
It suffices to show that 
 for any $i\in\N\cup\{0\}$
\begin{equation}\label{eq-gradient-parabola}
   \sG^{K^{i}}\left(\overline r_{k_R}^2u; B_{7R}(x_0)\right)\subset  \left\{x\in B_{7R}(x_0): R|\D u(x)| \leq C K^{i}\right\},
\end{equation} 
 since the corollary    follows   from  the   same  argument   as in  Lemma \ref{lem-W2p}. % using     Proposition \ref{prop-decay-est-supersol}.   %  the same procedure as above
      %Once \eqref{eq-gradient-parabola} is proved,   we use Proposition \ref{prop-decay-est-supersol} to obtain the result  arguing in a similar way  as  Corollary \ref{cor-W2p}.  
%So we will prove \eqref{eq-gradient-parabola}.  
Fix $i\in\N\cup\{0\},$ and   $x\in \sG^{K^i}\left(\overline r_{k_R}^2u; B_{7R}(x_0)\right).$ Then 
there exist  $ q_1\in \cQ_{l_{1}}\left(B_{7R}(x_0)\right),$ and $q_2\in \cQ_{l_{2}}\left(B_{7R}(x_0)\right)$ for $1\leq l_1,l_2\leq K^{i}$    
  satisfying \eqref{eq-contact-pt-above-below}. 
%  \begin{equation}\label{eq-contact-pt-above-below}
%\begin{split}
%-  q_1+  q_1(x)%&:=-\sum_{j=1}^{l_-}  r_j^2v(\cdot;     r_j;  z_j) -\sum_{j=1}^{l_-} \frac{1}{2} d_{     y_j}^2 +\sum_{j=1}^{l_-}     r_j^2v(x;      r_j;  z_j) +\sum_{j=1}^{l_-} \frac{1}{2}d_{     y_j}^2(x)\\ 
%&   \overline r_{k_R}^2u -\overline r_{k_R}^2u(x)% \geq -  q_1+  q_1(x)
%&\leq \sum_{j=1}^{l_+} \overline r_j^2v(\cdot; \overline r_j;\overline z_j) +\sum_{j=1}^{l_+}\frac{1}{2}d_{\overline y_j}^2 -%\sum_{j=1}^{l_+} \overline r_j^2v(x; \overline r_j;\overline z_j) -\sum_{j=1}^{l_+}\frac{1}{2} d_{\overline y_j}^2(x) \\
%\leq   q_2-  q_2(x) \qquad\mbox{in $B_{7R}(x_0)$}.
 %\end{split}
%\end{equation}
    We recall from Lemma \ref{lem-hess-dist-sqrd} and Lemma \ref{lem-barrier} that   for any  
 $B_{7r}(z)\subset B_{7R}(x_0) $ and $  y\in \overline  B_{7 r}(z),$
%for all $j=1,\cdots, l_-$ and $k=1,\cdots, l_+$
\begin{align*}
\frac{1}{R}|\D d^2_{y}| = \frac{2}{R} d_{ y}< 28  &\quad\mbox{in $B_{7R}(x_0)\setminus \Cut(y),$}\\
\frac{1}{7R}  r^2|\D v(\cdot; r; z) |< C  &\quad\mbox{in $B_{7R}(x_0)\setminus  \Cut(z)$}
\end{align*}
 for  a uniform constant $C>0$ depending only on $n,\lambda,\Lambda,$ and $\sqrt{\kappa}R_0.$ 
 %where $B_{7\underline r_j}(\underline z_j)\subset B_{7R}(x_0)\subset B_{R_0}(x_0),$ and $\underline y_j\in \overline  B_{7\underline r_j}(\underline z_j)$  for any $j=1,2,\cdots,l_-.$ 
According to Lemma \ref{lem-contact-set-tech} and \eqref{eq-rep-q_1}, we see that  
$$x\not\in \bigcup_{j=1}^{l_1}\Cut(  z_j)\cup \bigcup_{j=1}^{l_1}\Cut(  y_j),$$%\cup  \bigcup_{j=1}^{l_+}\Cut(\underline z_j)\cup \bigcup_{j=1}^{l_+}\Cut(\underline y_j) $$ 
and hence %from      \eqref{eq-contact-pt-above-below}, %  for a unit vector $X\in T_xM,$
$$\frac{\delta_0^2 }{4}R| \D u(x)| \leq \frac{\overline r_{k_R}^2}{R}| \D u(x)| \leq C l_1\leq C K^{i}$$
 since $\overline r_{k_R}\in \left(\frac{\delta_0 R}{2},\frac{R}{2}\right]$.  
 %$$\frac{\overline r_{k_R}^2}{R}\left\langle \D u(x), X\right\rangle \geq -C l_1\geq -C K^{i}.$$
%According to our argument above  with    \eqref{eq-contact-pt-above-below}, 
Therefore, it follows   that 
$$R|\D u(x)|\leq C K^{i}\quad \forall x\in\sG^{K^i}\left(\overline r_{k_R}^2u; B_{7R}(x_0)\right),$$ 
%  since $\overline r_{k_R}\in \left(\frac{\delta_0 R}{2},\frac{R}{2}\right]$. This 
which completes the proof of \eqref{eq-gradient-parabola}.
%if $x\in \sG^K(R^2u), $ then $R|\D u(x)|\leq CK.$ In fact,  there exists  a concave  paraboloid $-\frac{h}{2}d_y^2+C$ for some $y\in\overline B_{7R}(z_0),$ which   touches $R^2u$ from below at $x.$ Then we have $R^2|\D u(x)|= h d_{y}(x)\leq 7R h,$ proving $R|\D u(x)|\leq 7h.$
\end{proof}

{\bf Proof of Theorem \ref{thm-W2epsilon}} According to \cite[Theorem 1.5]{WZ}, we have the weak Harnack inequality which provides a uniform $L^\varepsilon$-estimate  of $u.$  
%We may assume that $u\not\equiv0,$ $$\|u\|_{L^{\infty}(B_{2R}(x_0))}\leq1/2\quad\mbox{and}\quad\left(\fint_{B_{2R}(x_0)} |R^2 f|^{n\eta}\right)^{\frac{1}{n\eta}}\leq \e/2$$ replacing $u$ by $ {u}/\left\{2\|u\|_{L^{\infty}(B_{2R}(x_0))}+  2\left(\fint_{B_{2R}(x_0)} |R^2 f|^{n\eta}\right)^{\frac{1}{n\eta}}/\epsilon\right\},$ where $\epsilon>0$ is the constant in Lemma \ref{lem-W2p}.  
By  passing  from cubes to balls with the help of Bishop-Gromov's Theorem \ref{thm-BG}, we conclude a uniform $W^{2,\ve}$-estimate;  refer to  Remark 8.3 and   Theorem 8.1 of  \cite{Ca} for  details.   
\qed 
\begin{comment}
% Arguing in a similar way in the proof of  \cite[Theorem 8.1]{Ca} with the help of the previous results,  
\begin{thm}[$W^{2,\ve}$-estimate] 
Let $0<R\leq R_0,$   $x_0\in M$ and $f\in L^{n\eta}\left(B_{2R}(x_0)\right)$ for $ \eta:=1+\log_2\cosh (4\sqrt\kappa R_0).$  
There exists   uniform constant $\ve\in(0,1)$ and $C>0$ such that for   a smooth, bounded solution  $u\in \cS^*(\lambda,\Lambda,f)$ in % solution to $F(D^2u)=f$ in 
$B_{2R}(x_0),$  
we have that   $u\in W^{2,\ve}\left(B_R(x_0)\right)$ with the estimate   
$$\left(\fint_{B_{R}(x_0)}  |u|^\ve+ \left|R \D u\right|^{\ve}+ \left|R^2D^2u\right|^\ve  \right)^{\frac{1}{\ve}} \leq C \left\{ \|u\|_{L^{\infty}\left(B_{2R}(x_0)\right)}+\left(\fint_{B_{2R}(x_0)} |R^2f|^{n\eta}\right)^{\frac{1}{n\eta}}\right\},$$
where the uniform constants  $\ve\in(0,1)$ and $C>0$ depend only on $n,\lambda,\Lambda,$ and $\sqrt{\kappa}R_0.$ 
\end{thm}
\end{comment}

%%%%%%%%%%%%%%%%%%%%%%%%%%%%%%%%%%%%%%%%%%%%%%%

%------------------------------------------------------------------------------%

\end{document}